\documentclass[letterpaper]{amsart}

\usepackage{amsthm,enumerate,amssymb}
\usepackage{amsmath,amscd}
\usepackage{cleveref}

\makeatletter
\newtheorem*{rep@theorem}{\rep@title}
\newcommand{\newreptheorem}[2]{%
\newenvironment{rep#1}[1]{%
 \def\rep@title{#2 \ref{##1}}%
 \begin{rep@theorem}}%
 {\end{rep@theorem}}}
\makeatother

\makeatletter
\newtheorem*{rep@cor}{\rep@ctitle}
\newcommand{\newrepcor}[2]{%
\newenvironment{rep#1}[1]{%
 \def\rep@ctitle{#2 \ref{##1}}%
 \begin{rep@cor}}%
 {\end{rep@cor}}}
\makeatother

\newtheorem{thm}{Theorem}[section]
\newtheorem{lemma}[thm]{Lemma}
\newtheorem{prop}[thm]{Proposition}
\newtheorem{cor}[thm]{Corollary}

\newtheorem{theorem}[thm]{Theorem}
\newreptheorem{theorem}{Theorem}
\newrepcor{cor}{Corollary}

\newtheorem*{thma}{Theorem A}
\newtheorem*{thmb}{Theorem B}
\newtheorem*{thmc}{Theorem C}
\newtheorem*{thmd}{Theorem D}
\newtheorem*{core}{Corollary E}
\newtheorem*{corf}{Corollary F}

\theoremstyle{definition}
\newtheorem{definition}[thm]{Definition}
\newtheorem{example}[thm]{Example}
\newtheorem{examples}[thm]{Examples}

\theoremstyle{remark}
\newtheorem{remark}[thm]{Remark}

\numberwithin{equation}{section}

\newcommand{\R}{\mathbb{R}}
\newcommand{\N}{\mathbb{N}}
\newcommand{\Z}{\mathbb{Z}}
\newcommand{\eps}{\epsilon}
\newcommand{\diam}{\text{diam}}

\newcommand\mmeas[1]{\smash{\overline{d}}\vphantom{d}^{#1}}
\newcommand\mmeasempty{\overline{d}}
\newcommand\mdim{\smash{\overline{D}}\vphantom{D}}
\newcommand\cmeas[1]{d^{#1}}
\newcommand\cmeasempty{d}
\newcommand\cdim{D}
\newcommand\covmeas[1]{d_{H}^{#1}}
\newcommand\covdim{D_{H}}

\newcommand{\hdist}{\mathop{\mathrm{dist_H}}}

\newcommand{\pmat}{P_M}
\newcommand{\pone}{\pi}
\newcommand{\ptwo}{\pi'}
\newcommand{\compact}{\mathcal{M}}
\newcommand{\vol}{\text{vol}}

\newcommand\matspace[2]{M_{{#1} \times {#2}}}

\newcommand\matmap[2]{[{#1}]_{#2}}
\newcommand\matmape[1]{[{#1}]}
\newcommand\restmatmap[2]{\overline{[{#1}]_{#2}}}
\newcommand\restmatmape[1]{\overline{[{#1}]}}

\newcommand\purerestmatmap[1]{\overline{[\, \cdot \,]_{#1}}}
\newcommand\purerestmatmape{\overline{[\, \cdot \,]}}

\newcommand\linmaps{\mathcal{L}}
\newcommand\nullspace[1]{\text{Null(}{#1}\text{)}}
\newcommand\imagespace[1]{\text{Range(}{#1}\text{)}}
\newcommand\rangespace[1]{\text{Range(}{#1}\text{)}}
\newcommand\proj[2]{\mathcal{P}^{#1}_{#2}}
\newcommand\projm[2]{\mu_{{#1},{#2}}}

\newcommand\oproj[2]{\mathcal{U}^{#1}_{#2}}
\newcommand\oprojm[2]{\nu^{#1}_{#2}}

\newcommand\goo{G}

\newcommand\mynextmap[2]{\rho}
\newcommand\myrestnextmap[2]{\overline{\rho}}
\newcommand\obtoorth{\varphi}

\newcommand\repnum[3]{R_{#1,#2}(#3)}
\newcommand\repnumsum[2]{S_{#1}(#2)}

\newcommand\restr[2]{{
  \left.\kern-\nulldelimiterspace 
  #1 
  \right|_{#2} 
  }}

\newcommand\Is{\mathcal{I}}

\title[Marstrand-type Theorems in $\Z^d$]{Marstrand-type Theorems for the Counting and Mass Dimensions in $\Z^d$}

\author{Daniel Glasscock}
\address{Department of Mathematics, The Ohio State University, 231 W. 18th Ave.,
Columbus, OH 43210}
\email{glasscock.4@math.osu.edu}
\urladdr{people.math.osu.edu/~glasscock.4}

\begin{document}

\begin{abstract}
The counting and (upper) mass dimensions of a set $A \subseteq \R^d$ are
\[\cdim(A) = \limsup_{\|C\| \to \infty} \frac{\log \big| \lfloor A \rfloor \cap C \big|}{\log \|C\|}, \quad \mdim(A) = \limsup_{\ell \to \infty} \frac{\log \big| \lfloor A \rfloor \cap [-\ell,\ell)^d \big|}{\log (2 \ell)},\]
where $\lfloor A \rfloor$ denotes the set of elements of $A$ rounded down in each coordinate and where the limit supremum in the counting dimension is taken over cubes $C \subseteq \R^d$ with side length $\|C\| \to \infty$. We give a characterization of the counting dimension via coverings: $\cdim(A) = \inf \{ \alpha \geq 0 \ | \ \covmeas\alpha(A) = 0 \}$ where
\[\covmeas\alpha(A) = \lim_{r \rightarrow 0} \limsup_{\|C\| \rightarrow \infty} \inf \left\{ \ \sum_i \left( \frac{\|C_i\|}{\|C\|} \right)^\alpha \ \middle| \ 1 \leq \|C_i\| \leq r \|C\| \right\}\]
in which the infimum is taken over cubic coverings $\{C_i\}$ of $A \cap C$. Then we prove Marstrand-type theorems for both dimensions. For example, almost all images of $A \subseteq \R^d$ under orthogonal projections with range of dimension $k$ have counting dimension at least $\min \big(k, \cdim(A) \big)$; if we assume $\cdim(A) = \mdim(A)$, then the mass dimension of $A$ under the typical orthogonal projection is equal to $\min \big(k, \cdim(A) \big)$. This work extends recent work of Y. Lima and C. G. Moreira.
\end{abstract}

\maketitle
\tableofcontents

\section{Introduction}

Notions of dimension for subsets of $\Z$ and $\Z^d$ have been studied by Naudts \cite{naudts1,naudts2}, Barlow and Taylor \cite{bat1,bat2}, Iosevich, Rudnev, and Uriarte-Tuero \cite{iosevich}, and, most recently, by Lima and Moreira \cite{lima}. Connections are made in each of these works to concepts and results from the traditional continuous theory of dimension. For example, Barlow and Taylor define analogues in the discrete setting to many of the classical dimensional quantities and describe the dimension of self-similar sets and random walks in the lattice $\Z^d$.

Lima and Moreira introduced the so-called counting dimension for subsets of $\Z$ and proved a Marstrand-type theorem for it. Marstrand's theorem is one of the fundamental theorems in geometric measure theory. Roughly speaking, Marstrand \cite{marstrand} showed that for a Borel set in the plane of Hausdorff dimension less than 1, almost all of its orthogonal projections have Hausdorff dimension equal to that of the original set.

The primary goal of the present work is to extend Lima and Moreira's counting dimension to subsets of $\Z^d$, give a characterization for it in terms of coverings, and strengthen their Marstrand-type result.

Let $A \subseteq \R^d$ and $\alpha \geq 0$. The $\alpha$-counting measures and the counting dimension of $A$ are defined to be
\[\cmeas \alpha(A) = \limsup_{\|C\| \to \infty} \frac{\big| \lfloor A \rfloor \cap C \big|}{\|C\|^\alpha}, \qquad \cdim(A) = \limsup_{\|C\| \rightarrow \infty} \frac{\log \big| \lfloor A \rfloor \cap C \big|}{\log \|C\|},\]
where $\lfloor A \rfloor$ denotes the set of elements of $A$ rounded down in each coordinate and where the limit supremum is taken over cubes $C \subseteq \R^d$ with side length $\|C\| \to \infty$. (This specializes to Lima and Moreira's definition of the counting dimension when $d = 1$ and $A \subseteq \Z$.) The counting measures and dimension are the base-point free versions of the (upper) mass measures and dimension, defined by
\[\mmeas \alpha(A) = \limsup_{\ell \to \infty} \frac{\big| \lfloor A \rfloor \cap [-\ell,\ell)^d \big|}{(2\ell)^\alpha}, \qquad \mdim(A) = \limsup_{\ell \to \infty} \frac{\log \big| \lfloor A \rfloor \cap [-\ell,\ell)^d \big|}{\log (2 \ell)}.\]
The 1-counting and 1-mass measures are the usual upper Banach density and upper density that measure the linear growth rate of a set on long intervals in $\Z$; sets of positive upper (Banach) density have important applications in combinatorics and ergodic theory (see, for example, \cite{bergelson} and \cite{furstenberg}). Sets of zero density may still be distinguished by differing rates of growth on long intervals by the $\alpha$-counting and $\alpha$-mass measures for $\alpha < 1$. Roughly speaking, the set $A$ exhibits $\cdim (A)^{-1}$-degree polynomial rate growth on some sequence of intervals with length tending to infinity.

The first main result is a characterization of the counting dimension via coverings. This characterization draws parallels to both the discrete Hausdorff dimension discussed in \cite{bat1,bat2} and the classical Hausdorff dimension. Specifically, for $A \subseteq \R^d$ and $\alpha \geq 0$, define
\[\covmeas \alpha(A) = \lim_{r \rightarrow 0} \limsup_{\|C\| \rightarrow \infty} \inf \left\{ \ \sum_i \left( \frac{\|C_i\|}{\|C\|} \right)^\alpha \ \middle| \ A \cap C \subseteq \bigcup_i C_i, \ 1 \leq \|C_i\| \leq r \|C\| \right\},\]
where the $C_i$'s are cubes in $\R^d$. We show that the resulting dimensional quantity $\covdim (A) = \inf \{\alpha \geq 0 \ | \ \covmeas\alpha(A) = 0 \}$ is equal to the counting dimension $\cdim(A)$.

\begin{thma}\label{thm:introequivalence}
For all $A \subseteq \R^d$, $\covdim(A) = \cdim(A)$.
\end{thma}

The second main result is a strengthening and generalization of the following Marstrand-type theorem of Lima and Moreira.  For $\lambda \in \R$ and $A, B \subseteq \R$, let $\lambda A = \{ \lambda a \ | \ a \in A\}$ and $A+B = \{a + b \ | \ a \in A, b \in B \}$.

\begin{thmb}[\cite{lima}]\label{thm:introlima}
Let $A, B \subseteq \Z$ be regular and compatible. For Lebesgue almost every $\lambda \in \R$,
\[\cdim \big(A + \lfloor \lambda B \rfloor \big) \geq \min \big( 1, \cdim(A) + \cdim(B) \big).\]
Moreover, if $\cdim(A) + \cdim(B) > 1$, then for Lebesgue almost every $\lambda \in \R$, the set $A + \lfloor \lambda B \rfloor$ has positive upper Banach density.
\end{thmb}

To see the connection with Marstrand's theorem, note that the images of the product set $A \times B \subseteq \R^2$ under (oblique) projections $\R^2 \to \R$ are exactly sets of the form $A + \lambda B$ for $\lambda \in \R$.

We prove the following generalization of Theorem B.

\begin{thmc}\label{thm:introobliquemarstrand}
Let $A \subseteq \R^d$ and $0 \leq k \leq d$. If $\alpha \neq k$ is such that $\cmeas \alpha (A) > 0$, then for almost every projection $P: \R^d \to \R^d$ with range $\R^k \times \{0\}^{d-k}$, $\cmeas {\min(k,\alpha)} (PA) > 0$. In particular, for almost every such projection $P$,
\[\cdim \big( PA \big) \geq \min \big(k, \cdim(A) \big)\]
and, if $\cdim(A) > k$, then $\cmeas k (PA) > 0$.
\end{thmc}

Complementing Theorem C, we give a concrete example showing that an increase in counting dimension is possible under the typical projection; that is, the inequality in Theorem C cannot be made an equality.  Specifically, we provide an example of a set $E \subseteq \Z$ of zero counting dimension such that for all $\lambda \neq 0$, the upper Banach density of the sumset $E + \lambda E$ is positive.

In the same spirit as Theorem C, we derive the following Marstrand-type theorem for the mass dimension. The definitions regarding regularity are given in Section \ref{sec:regularity}.

\begin{thmd}
Let $A \subseteq \R^d$ be such that $\cdim(A) = \mdim(A)$, and let $0 \leq k \leq d$. For almost every projection $P: \R^d \to \R^d$ with range $\R^k \times \{0\}^{d-k}$,
\[\mdim \big( PA \big) = \min \big(k, \mdim(A) \big)\]
and, if $A$ is counting and mass regular and $\mdim(A) \neq k$, then $\mmeas{\min (k, \mdim(A))}(P A) > 0$.
\end{thmd}

We then deduce from Theorems C and D related Marstrand-type results. For example, we prove the following orthogonal projection variant which is reminiscent in formulation of Mattila's generalization of Marstrand's original theorem.

\begin{core}
Let $A \subseteq \R^d$ be such that $\cdim(A) = \mdim(A)$, and let $0 \leq k \leq d$. For almost every orthogonal projection $U$ with range of dimension $k$,
\[\mdim(UA) = \min \big(k, \mdim(A) \big)\]
and, if $A$ is counting and mass regular and $\mdim(A) \neq k$, then $\mmeas{\min (k,\mdim(A))}(U A) > 0$.
\end{core}

Typicality in these results is with respect to the unique $O(d)$-invariant probability measure on the Grassmannian $G(d,k)$ under the association of a projection with its null space (in Theorems C and D) or with its range (in Corollary E).

The primary applications of Theorems C and D are similar to those obtained by Lima and Moreira and follow from the fact that projections of the product set $A_1 \times \cdots \times A_d \subseteq \R^d$ to $\R \times \{0\}^{d-1}$ are sets of the form $A_1 + \lambda_1 A_2 + \cdots + \lambda_{d-1} A_d$ for $\lambda \in \R^{d-1}$. Thus, Theorems C and D give us information on the counting and mass dimensions of the typical sumset of dilated sets. The following is an example of such an application.

\begin{corf}
For each $1 \leq i \leq d$, let $f_i \in \R[x]$ be a non-constant polynomial. For Lebesgue-almost every $\lambda \in \R^d$,
\[\mdim \big( \lambda_1 f_1(\Z) + \cdots + \lambda_d f_d(\Z) \big) = \min \left(1, \frac{1}{\deg f_1} + \cdots + \frac{1}{\deg f_d} \right).\]
Moreover, if $\sum_i (\deg f_i)^{-1} \neq 1$, then for almost every $\lambda \in \R^d$,
\[\mmeas{\min \left(1,\sum_i (\deg f_i)^{-1} \right)} \big(\lambda_1 f_1(\Z) + \cdots + \lambda_d f_d(\Z) \big) > 0.\]
\end{corf}

This paper is organized into two parts. In the first, we develop the basic properties of the counting and mass dimensions for subsets of $\R^d$ and prove Theorem A. In the second part, we prove Theorems C and D, derive from them additional Marstrand-type results (including Corollaries E and F), and give some applications.

\section{The counting and mass dimensions}

The upper mass dimension for subsets of $\Z^d$ was introduced by Barlow and Taylor in \cite{bat1, bat2}, and the counting dimension for subsets of $\Z$ was introduced by Lima and Moreira in \cite{lima}. In this section, we define the upper mass and counting dimensions for subsets of $\R^d$ and develop the basic properties. Many of these properties do not appear elsewhere in the literature, so in some cases for completeness we go beyond the material strictly necessary for the Marstrand-type theorems in the second part.

\subsection{Definitions and basic properties}

Let $\N = \{1,2,\ldots\}$ be the set of natural numbers and $\Z$ be the integers. All sequences in this work will be indexed by $\N$, and we will write $(x_n)_n$ as a shorthand for $(x_n)_{n\in\N}$. For a finite set $A$, denote by $|A|$ its cardinality.

\begin{definition}
A \emph{cube} in $\R^d$ will refer to a set of the form
\[C = \prod_{i=1}^d [a_i,a_i + \ell),\]
where $a = (a_1, a_2, \ldots, a_d) \in \R^d$ and $\ell > 0$. The cube $C$ is \emph{based at} $a$ and has side length $\|C\| = \ell$. The cube $C$ is \emph{centered} if it is symmetric about the origin.
\end{definition}

\begin{definition}
The floor $\lfloor a \rfloor$ of a real number $a$ is the greatest integer less than or equal to $a$. The floor function $\lfloor \, \cdot \, \rfloor$ is applied coordinate-wise to elements of $\R^d$ and element-wise to subsets of $\R^d$; in other words, for $A \subseteq \R^d$, $\lfloor A \rfloor$ denotes the set consisting of elements of $A$ ``rounded down'' in each coordinate.
\end{definition}

Now we may define the counting and (upper) mass dimensions.

\begin{definition}
Let $A \subseteq \R^d$ and $\alpha \geq 0$. The \emph{$\alpha$-counting measure} of $A$ is
\[\cmeas{\alpha,d}(A) = \limsup_{\| C \| \rightarrow \infty} \frac{\big| \lfloor A \rfloor \cap C \big|}{\|C\|^\alpha} := \lim_{\ell \rightarrow \infty} \sup_{\|C\| \geq \ell} \frac{\big| \lfloor A \rfloor \cap C \big|}{\|C\|^\alpha},\]
where the limit supremum and supremum are taken over cubes $C \subseteq \R^d$. The \emph{counting dimension} of $A$ is
\[\cdim^d(A) = \inf \left\{ \alpha \geq 0 \ \middle| \ \cmeas{\alpha,d}(A) = 0  \right\}.\]
\end{definition}

The counting dimension captures the maximal polynomial rate of growth on larger and larger cubes with respect to the cubes' side lengths. The (upper) mass dimension captures this maximal growth rate along centered cubes.

\begin{definition}
Let $A \subseteq \R^d$ and $\alpha \geq 0$. The \emph{$\alpha$-mass measure} of $A$ is
\[\mmeas{\alpha,d}(A) = \limsup_{\substack{\| C \| \rightarrow \infty \\ C \text{ centered}}} \frac{\big| \lfloor A \rfloor \cap C \big|}{\|C\|^\alpha} = \limsup_{\ell \to \infty} \frac{\big| \lfloor A \rfloor \cap [-\ell, \ell)^d \big|}{(2\ell)^\alpha}.\]
The \emph{(upper) mass dimension} of $A$ is
\[\mdim^d(A) = \inf \left\{ \alpha \geq 0 \ \middle| \ \mmeas{\alpha,d}(A) = 0  \right\}.\]
\end{definition}

\begin{remark}
It is immediate from the definitions that $\mmeas\alpha \leq \cmeas\alpha$ and $\mdim \leq \cdim$.
\end{remark}

Rounding down to the integer lattice $\Z^d$ is for mere computational convenience. In Section \ref{sec:invunderqie} we will show that the counting and mass dimensions are invariant under quasi-isometric embeddings. This will allow us to adopt a course geometry perspective and realize the counting and mass dimensions as measures of the rate of growth of sets ``at infinity.''

For brevity, we will drop the word ``upper'' and simply refer to the upper mass dimension as the mass dimension. When the ambient space is apparent (it will usually be $\R^d$), we will omit the dimension exponent and simply write $\cmeas\alpha$, $\mmeas\alpha$, $\cdim$, and $\mdim$ instead of $\cmeas{\alpha,d}$, $\mmeas{\alpha,d}$, $\cdim^d$, and $\mdim^d$. The exponent $\alpha$ will always be understood to be greater than or equal to $0$.

Finally, several of the results for the counting measures and dimension hold equally as well for the mass measures and dimension with minor modifications to the proofs. The phrase ``The same conclusions hold for the mass measures and dimension.'' means that the preceding statements hold with $\cmeasempty$ and $\cdim$ replaced by $\mmeasempty$ and $\mdim$, respectively.

\begin{lemma}\label{lem:basiccountingprops}
Let $A \subseteq \R^d$ and $\alpha \geq 0$. Then
\begin{enumerate}[i.]
\item for all $\alpha' \leq \alpha$, $\cmeas{\alpha'}(A) \geq \cmeas\alpha (A)$;
\item for all $A' \subseteq A$, $\cmeas\alpha(A') \leq \cmeas\alpha(A)$ and $\cdim(A') \leq \cdim(A)$;
\item $0 \leq \cdim(A) \leq d$;
\item if $\alpha < \cdim(A)$, then $\cmeas\alpha(A) = \infty$;
\item if $\alpha > \cdim(A)$, then $\cmeas\alpha(A) = 0$.
\end{enumerate}
If $A$ is non-empty, the counting dimension may be computed explicitly as
\begin{align}\label{eqn:explicitformula}\cdim(A) = \limsup_{\|C\| \rightarrow \infty} \frac{\log \big| \lfloor A \rfloor \cap C \big|}{\log \| C\|}.\end{align}

The same conclusions hold for the mass measures and dimension with the limit supremum in (\ref{eqn:explicitformula}) taken over centered cubes.
\end{lemma}

\begin{proof}
Properties \emph{i.} and \emph{ii.} are immediate from the definitions. Property \emph{iii.} follows from the fact that $\cmeas\alpha \equiv 0$ when $\alpha > d$. Property \emph{v.} is immediate from the definition of $\cdim$ and the monotonicity of $\cmeas\alpha(A)$ in $\alpha$ (property \emph{i.}).

To prove \emph{iv.}, it suffices by monotonicity in $\alpha$ to show that if $\alpha$ is such that $\cmeas{\alpha}(A) > 0$, then for all $0 \leq \alpha' < \alpha$, $\cmeas{\alpha'}(A) = \infty$. So, suppose $\cmeas{\alpha}(A) > 0$. There exists an $\eps > 0$ and a sequence of cubes $(C_n)_n$, $\|C_n\| \rightarrow \infty$, such that
\[\lim_{n \rightarrow \infty} \frac{\big| \lfloor A \rfloor \cap C_n \big|}{\|C_n\|^{\alpha}} > \eps.\]
For all $0 \leq \alpha' < \alpha$, since $\lim_{n \rightarrow \infty} \|C_n\|^{\alpha - \alpha'} = \infty$,
\[\cmeas{\alpha'}(A) = \limsup_{\|C\| \rightarrow \infty} \frac{\big| \lfloor A \rfloor \cap C \big|}{\|C\|^{\alpha'}} \geq \limsup_{n \rightarrow \infty} \|C_n\|^{\alpha - \alpha'}\frac{\big| \lfloor A \rfloor \cap C_n \big|}{\|C_n\|^{\alpha}} = \infty.\]
To prove (\ref{eqn:explicitformula}), let $\alpha < \cdim(A)$, and pick, by \emph{iv.}, a sequence of cubes $(C_n)_n$, $\|C_n\| \rightarrow \infty$, such that for all $n \in \N$, $\big| \lfloor A \rfloor \cap C_n \big| \big/ \|C_n\|^\alpha \geq 1$. Then
\[\limsup_{\|C\| \rightarrow \infty} \frac{\log \big| \lfloor A \rfloor \cap C \big|}{\log \| C\|} \geq \limsup_{n \rightarrow \infty} \frac{\log \big| \lfloor A \rfloor \cap C_n \big|}{\log \| C_n\|} \geq \alpha.\]
Since $\alpha < \cdim(A)$ was arbitrary $\limsup_{\|C\| \rightarrow \infty} \frac{\log | \lfloor A \rfloor \cap C |}{\log \| C\|} \geq \cdim(A)$. For the reverse inequality, let $\alpha > \cdim(A)$. It follows from \emph{v.} that there exists a $0 < K < \infty$ such that for all cubes $C$, $\big| \lfloor A \rfloor \cap C \big| \leq K \|C\|^\alpha$. It follows that
\[\limsup_{\|C\| \rightarrow \infty} \frac{\log \big| \lfloor A \rfloor \cap C \big|}{\log \| C\|} \leq \alpha.\]
Since $\alpha > \cdim(A)$ was arbitrary, $\limsup_{\|C\| \rightarrow \infty} \frac{\log | \lfloor A \rfloor \cap C |}{\log \| C\|} \leq \cdim(A)$.
\end{proof}

It will often be convenient to choose a specific sequence of cubes along which a set achieves its counting or mass dimension.

\begin{definition}\label{def:achievesdimension}
A non-empty set $A \subseteq \R^d$ \emph{achieves its counting dimension along the sequence of cubes} $(C_n)_n \subseteq \R^d$, $\|C_n\| \rightarrow \infty$, if
\[\lim_{n \rightarrow \infty} \frac{\log \big| \lfloor A \rfloor \cap C_n \big|}{\log \|C_n\|} = \cdim(A).\]
If the cubes $(C_n)_n$ are centered and the limit equals $\mdim(A)$, then $A$ \emph{achieves its mass dimension along $(C_n)_n$}.
\end{definition}

We conclude this section by showing that the counting and mass measures are finitely sub-additive.

\begin{lemma}\label{lem:countingprops}
Let $A_1, \ldots, A_m \subseteq \R^d$. Then
\[\cmeas\alpha \left( \bigcup_{i=1}^m A_i \right) \leq \sum_{i=1}^m \cmeas\alpha(A_i) \quad \text{and} \quad \cdim \left( \bigcup_{i=1}^m A_i \right) = \max_{1 \leq i \leq m} \cdim(A_i).\]
The same conclusions hold for the mass measures and dimension.
\end{lemma}

\begin{proof} It suffices to prove the statements for two sets $A, A' \subseteq \R^d$. Since $\lfloor A \cup A' \rfloor$ is $\lfloor A \rfloor \cup \lfloor A' \rfloor$, we have
\begin{align*}
\cmeas\alpha(A \cup A') &= \limsup_{\|C\| \rightarrow \infty} \frac{\big| \lfloor A \cup A' \rfloor \cap C \big|}{\|C\|^\alpha} \\
&= \limsup_{\|C\| \rightarrow \infty} \frac{\big|(\lfloor A \rfloor \cup \lfloor A' \rfloor) \cap C \big|}{\|C\|^\alpha} \\
&\leq \limsup_{\|C\| \rightarrow \infty} \frac{\big|\lfloor A \rfloor \cap C \big|}{\|C\|^\alpha} + \limsup_{\|C\| \rightarrow \infty} \frac{\big|\lfloor A' \rfloor \cap C \big|}{\|C\|^\alpha}\\
&= \cmeas\alpha(A) + \cmeas\alpha(A').
\end{align*}

It follows from the definition of $\cdim$ and the subadditivity of $\cmeas\alpha$ that
\begin{align*}
\cdim(A \cup A') & = \inf \big\{ \alpha \in \R \ \big| \ \cmeas\alpha(A \cup A') = 0 \big\}\\
& \leq \inf \big\{ \alpha \in \R \ \big| \ \cmeas\alpha(A) = 0 \text{ and } \cmeas\alpha(A') = 0\big\} \\
& = \max \big( \inf \big\{ \alpha \in \R \ \big| \ \cmeas\alpha(A) = 0 \big\}, \inf \big\{ \alpha \in \R \ \big| \ \cmeas\alpha(A') = 0 \big\} \big)\\
& = \max \big( \cdim(A), \cdim(A') \big).
\end{align*}
\end{proof}

\subsection{Examples}\label{sec:examples}

Here we collect some examples of sets and their counting and mass dimensions. Examples i. - v. are in $\R$.

\begin{enumerate}[i.]
\item Any set with positive upper Banach density has counting dimension 1, while any set with positive upper density has mass dimension 1.  There are sets of zero upper density of mass dimension 1; by the prime number theorem, the set of prime numbers is such. If $A$ contains arbitrarily long intervals, i.e. if $A$ is thick, then clearly $\cdim(A) = 1$; it is easy to construct thick sets which are of zero mass dimension. 
\item Let $f: \R \rightarrow \R$ be a real polynomial of degree $n \geq 1$. The image of $\Z$ under $f$ has counting and mass dimension $1/n$, as can be directly computed (or, see \cite{lima}). Along the same lines, for $\beta > 0$, it is straightforward to check that
\[\cdim\big(\{n^\beta \ | \ n \in \N\}\big) = \mdim\big(\{n^\beta \ | \ n \in \N\}\big) = \min \big(\beta^{-1},1 \big).\]
If $A = \{0 < a_1 < a_2 < \cdots \}$ is such that $a_n \leq n^\beta$ for infinitely many $n$, then $\mdim(A) \geq \min \big(\beta^{-1},1\big)$. Therefore, if $\mdim(A) < \alpha \leq 1$, there exists a $\beta > \alpha^{-1}$ such that $a_n > n^\beta$ eventually; in particular, this implies that $\sum a_n^{-\alpha} < \infty$. The full converse does not hold: the sequence $\big(\big( n (\log n)^2 \big)^{1 / \alpha}\big)_n$ has mass dimension $\alpha$ and is such that $\sum a_n^{-\alpha}$ converges. For a partial converse, note that if $\mdim(A) > \alpha$, then $\big|A \cap [-\ell, \ell) \big| \big/ (2 \ell)^\alpha \to \infty$ as $\ell \to \infty$, implying that $\sum a_n^{-\alpha} = \infty$. This implies that
\[\mdim(A) = \inf \left\{ \alpha \geq 0 \ \middle| \ \sum a_n^{-\alpha} < \infty \right\}.\]
There is no such statement for the counting dimension since the counting dimension of $A$ is unrelated to the concentration of $A$ about the origin.
\item For any $r > 1$, the geometric sequence $\big(r^n\big)_{n}$ has zero counting and mass dimension. In fact, both dimensions are 0 for any lacunary sequence.
\item \label{item:integercantor} (Following \cite{lima}) Let $E$ be those positive integers that may be written in base 3 using only the digits 0 and 1. The set $E$ has counting and mass dimension $\log 2 \big/ \log 3$.  More generally, let $b \geq 2$, and let $M = (m_{i,j})_{0 \leq i,j \leq b-1}$ be a $b \times b$, binary transition matrix. Consider the \emph{integer Cantor set}
\[E_M = \big\{d_0 b^0 + \cdots + d_n b^n \ | \ n \geq 0, \ m_{d_{i-1},d_i} = 1 \text{ for all } 1 \leq i \leq n \big\}.\]
The counting and mass dimension of $E_M$ are
\[\cdim(E_M) = \mdim(E_M) = \frac{\log \lambda_+(M)}{\log b},\]
where $\lambda_+(M)$ denotes the Perron-Frobenius eigenvalue of $M$. Even more generally, for a closed, left-shift invariant subset $X \subseteq \{0, 1,\ldots, b-1\}^{\N}$, consider the set
\[E_X = \big\{\omega_1 b^0 + \cdots + \omega_{n+1} b^{n} \ | \ n \geq 0, \ \omega \in X \big\}.\]
It follows from calculations very similar to those in Lima and Moreria and the connection between the exponential growth rate of words in $X$ to $h_{\text{top}}(X)$, the topological entropy of $X$, that
\[\cdim(E_X) = \mdim(E_X) = \frac{h_{\text{top}}(X)}{\log b}.\]
\item \label{item:genipset} (Following \cite{lima}) Given two sequences $(k_n)_n$ and $(d_n)_n$ of positive integers satisfying $d_{n+1} > \sum_{i=1}^n k_i d_i$, consider the associated \emph{generalized IP set}
\[E = \left\{ \sum_{i=1}^n x_i d_i \ \middle| \ n \in \N, 0 \leq x_i < k_i \right\}.\]
It can be shown that the counting and mass dimensions of $E$ are
\[\cdim(E) = \mdim(E) = \limsup_{n \to \infty} \frac{\log (k_1 \cdots k_n) }{\log (k_n d_n)}.\]
\item The counting and mass dimensions are preserved under quasi-isometric embeddings (see Section \ref{sec:invunderqie}). For example, if $A = \{a_1 < a_2 < \cdots \} \subseteq \Z$, then the map $a_n \mapsto (n, a_n)$ from $A$ to its graph $G \subseteq \Z^2$ is a quasi-isometric embedding (actually, the map is bi-Lipschitz). Therefore, $\cdim(G) = \cdim(A)$ and $\mdim(G) = \mdim(A)$.
\item Examples in higher dimensions may be obtained by taking Cartesian products (see Section \ref{sec:products}). There is an important notion of compatibility between sets introduced in \cite{lima}; see Definition \ref{def:compatible}. In short, any two sets $A, B \subseteq \R$ satisfy
\[\max \big( \cdim(A),\cdim(B) \big) \leq \cdim(A \times B) \leq \cdim(A) + \cdim(B),\]
with equality on the right hand side if and only if $A$ and $B$ are counting compatible. The same inequality holds for $\mdim$. For example, if $f$ is a real, non-constant polynomial and $E$ is the integer Cantor set from example \ref{item:integercantor}., then the counting and mass dimensions of $f(\Z) \times E$ in $\R^2$ are $(\deg f)^{-1} + \log 2 / \log 3$.
\end{enumerate}

Barlow and Taylor \cite{bat2} consider the upper mass dimension (among other dimensional quantities) of self-similar sets in $\R^d$ and of random walks on the lattice $\Z^d$. The reader is encouraged to consult their work for further examples regarding the upper mass dimension.

\subsection{Regularity and regular subsets}\label{sec:regularity}

An important role in the traditional continuous theory of dimension is played by $s$-sets, those with non-zero, finite $s$-Hausdorff measure. The analogous notion of regularity in this setting is developed below; it was defined and used in the one-dimensional case by Lima and Moreira in \cite{lima}, and we follow their terminology.

\begin{definition}
A set $A \subseteq \R^d$ is \emph{$\alpha$-counting regular} if $0 < \cmeas\alpha (A) < \infty$; more succinctly, the set $A$ is \emph{counting regular} if $0< \cmeas{\cdim(A)}(A) < \infty$. Similarly, the set $A$ is \emph{$\alpha$-mass regular} if $0 < \mmeas\alpha (A) < \infty$ and simply \emph{mass regular} if $0< \mmeas{\mdim(A)}(A) < \infty$.
\end{definition}

A fundamental fact is that a set $A \subseteq \R^d$ contains $\alpha$-counting regular subsets for every $\alpha < \cdim(A)$. The following proposition is best seen as an analogue to the well known fact that sets of Hausdorff dimension greater than $\alpha$ contain subsets with finite, positive $\alpha$-Hausdorff measure (see Mattila \cite{mattilabook}, Chapter 8).

\begin{prop}\label{prop:regularsubset}
Let $A \subseteq \R^d$. If $\cmeas\alpha(A) = \infty$, then there exists an $\alpha$-counting regular subset of $A$. In particular, for all $0 \leq \alpha < \cdim(A)$, there exists an $\alpha$-counting regular subset of $A$.
\end{prop}

The case $d=1$ was proven by Lima and Moreira. The proof below is similar in spirit and goes as follows. First, we transform this into a problem about finite sets by passing to a subset of $A$ which consists of pieces that are sufficiently distant or ``disjoint.'' We then thin this subset on each piece separately in a controlled manner to achieve the desired dimension. We begin with two lemmas.

\begin{lemma}\label{lem:disjointcubes}
Let $A \subseteq \R^d$. If $\cmeas\alpha(A) = \infty$, then there exists a sequence of pairwise disjoint cubes $(C_n)_n$, $\|C_n\| \rightarrow \infty$, for which
\begin{align}\label{eqn:limisinfinity} \lim_{n\rightarrow \infty} \frac{\big|\lfloor A \rfloor \cap C_n\big|}{\|C_n\|^\alpha} = \infty.\end{align}
\end{lemma}

\begin{proof}
Let $C = \prod_{i=1}^d [a_i,b_i)$ be a cube in $\R^d$. For each $1 \leq i \leq d$, consider the partition $\big(-\infty,\lfloor a_i \rfloor \big) \cup \big[ \lfloor a_i \rfloor , \lceil b_i \rceil \big) \cup \big[ \lceil b_i \rceil, \infty \big)$ of $\R$; these partitions generate a partition of $\R^d$ into $3^d$ pieces, one of which is bounded and contains the cube $C$.

By Lemma \ref{lem:countingprops}, one of the pieces $P$ of this partition satisfies $\cmeas\alpha(A \cap P) = \infty$. Since the $\alpha$-counting measure of any bounded set is finite, $P$ must be one of the $3^d-1$ unbounded pieces, all of which are disjoint from $C$. Note that $P$ is aligned with $\Z^d$ in the sense that $\lfloor A \rfloor \cap P = \lfloor A \rfloor \cap \lfloor P \rfloor = \lfloor A \cap P \rfloor$.  Finally, if $C_0' \subseteq \R^d$ is a cube, then it is not hard to see that there exists a translate $C_0$ of $C_0'$ disjoint from $C$ which contains $C_0' \cap P$; that is, $\|C_0\| = \|C_0'\|$, $C_0 \cap C = \emptyset$, and $C_0 \supseteq C_0' \cap P$.

Now, by the definition of $\cmeas\alpha(A)$, we may choose a cube $C_1$, $\|C_1\| \geq 2$, such that
\[\frac{\big| \lfloor A \rfloor \cap C_1 \big|}{\|C_1\|^\alpha} \geq 2.\]
Let $P_1 \subseteq \R^d$ be an unbounded piece of the partition corresponding to $C_1$ described above with the property that $\cmeas\alpha(A \cap P_1) = \infty$. By the definition of $\cmeas\alpha(A \cap P_1)$, we may choose a cube $C_2'$, $\|C_2'\| \geq 2^2$, such that $\big| \lfloor A \cap P_1 \rfloor \cap C_2' \big| \big/ \|C_2'\|^\alpha$ is at least $2^2$. By the remark above, there exists a cube $C_2$, $\|C_2\| = \|C_2'\|$, disjoint from $C_1$, for which $C_2 \supseteq C_2' \cap P_1$. It follows that
\[\frac{ \big| \lfloor A \rfloor \cap C_2 \big|}{ \| C_2\|^\alpha} \geq \frac{ \big| \lfloor A \rfloor \cap P_1 \cap C_2' \big|}{ \| C_2'\|^\alpha} = \frac{ \big| \lfloor A \cap P_1 \rfloor \cap C_2'\big|}{ \| C_2'\|^\alpha} \geq 2^2.\]

Suppose now that $k \geq 2$ and the cubes $C_j$, $1 \leq j \leq k$, have been defined. Let $C$ be a cube containing $C_1 \cup \cdots \cup C_k$. Let $P_k \subseteq \R^d$ be an unbounded piece of the partition corresponding to the cube $C$ with the property that $\cmeas\alpha(A \cap P_k) = \infty$. By the argument given above, there exists a cube $C_{k+1}$,  $\|C_{k+1}\| \geq 2^{k+1}$, disjoint from $C$ (and thus from each $C_j$, $1 \leq j \leq k$) for which
\[\frac{ \big| \lfloor A \rfloor \cap C_{k+1} \big|}{\| C_{k+1}\|^\alpha} \geq \frac{ \big| \lfloor A \cap P_k \rfloor \cap C_{k+1} \big|}{\| C_{k+1}\|^\alpha} \geq 2^{k+1}.\]

This defines inductively a sequence of pairwise disjoint cubes $(C_n)_n$, $\|C_n\| \rightarrow \infty$, satisfying (\ref{eqn:limisinfinity}).
\end{proof}

\begin{remark}
It can be shown in the same way that for any unbounded $A \subseteq \R^d$, there exists a sequence of \emph{disjoint} cubes along which $A$ achieves its counting dimension.
\end{remark}

\begin{lemma}\label{lem:techlemredo}
Let $E \subseteq \Z^d$ be a finite set contained in a cube $C_E \subseteq \R^d$, $\|C_E\| \geq 1$. Suppose that $\alpha > 0$ and that
\[S:= \frac{|E \cap C_E|}{\|C_E\|^\alpha} \geq 6^d.\]
There exists an $F \subseteq E$ and a cube $C_{F} \supseteq F$ with side length $\|C_{F}\| \geq \sqrt[d]{S}/6$ satisfying
\[ 2^d \leq \sup_{\|C\| \geq 1} \frac{|F \cap C|}{\|C\|^\alpha} = \frac{|F \cap C_{F}|}{\|C_{F}\|^\alpha} < 2^d+1.\]
\end{lemma}

\begin{proof}
Set $\ell = \sqrt[d]{S}/6$. Note that $1 \leq \ell \leq \|C_E\| \big/ 2$ since $S \geq 6^d$ and $\|C_E\| \geq 1$. Partition $C_E$ evenly into $\lfloor \|C_E\| / \ell \rfloor^d$ sub-cubes $\{C_{E,i}\}_{i \in I}$, each of side length $\|C_E\| \big / \lfloor \|C_E\| / \ell \rfloor$. The side length of the sub-cubes satisfies $\ell \leq \|C_E\| \big / \lfloor \|C_E\| / \ell \rfloor \leq 2\ell$. Choose a subset $E' \subseteq E$ in the following way: for each $i \in I$, if $E \cap C_{E,i}$ is non-empty, choose exactly one point of $E$ in $C_{E,i}$.

For each $i \in I$, since $\big|\Z^d \cap C_{E,i} \big| \leq (3\ell)^d$,
\[|E' \cap C_{E,i}| \geq \frac{|E \cap C_{E,i}|}{(3\ell)^d}.\]
Summing the previous inequality over the smaller cubes,
\[\sup_{\|C\| \geq 1} \frac{|E' \cap C|}{\|C\|^\alpha} \geq \frac{|E' \cap C_E|}{\|C_E\|^\alpha} \geq \frac{S}{(3\ell)^d} = 2^d.\]

Note that for $e' \in E'$,
\[0 \leq \sup_{\|C\| \geq 1} \frac{|E' \cap C|}{\|C\|^\alpha} - \sup_{\|C\| \geq 1} \frac{\big|(E' \setminus \{e'\}) \cap C \big|}{\|C\|^\alpha}\leq 1,\]
which is to say that removing one element from $E'$ decreases the respective supremum by at most 1. Therefore, since $E'$ is finite, we may remove elements successively from $E'$ to arrive at an $E'' \subseteq E'$ satisfying
\[2^d \leq \sup_{\|C\| \geq 1} \frac{|E'' \cap C|}{\|C\|^\alpha} < 2^d +1.\]

Finally, let $C_F$ be a cube with $\|C_F\| \geq 1$ realizing this supremum and set $F = E'' \cap C_F$. We have only to show that $\|C_F\| \geq \ell$; if this were not the case, then $C_F$ could intersect at most $2^d$ of the sub-cubes $C_{E,i}$, whereby $|F \cap C_F| \leq 2^d$ and $|F \cap C_F| \big/ \|C_F\|^{\alpha} < 2^d$, a contradiction.
\end{proof}

We can now prove Proposition \ref{prop:regularsubset}.

\begin{proof}
It suffices to prove the assertion for $A \subseteq \Z^d$. Indeed, suppose $A \subseteq \R^d$ and $\cmeas\alpha(A) = \infty$. By definition, $\cmeas\alpha \big(\lfloor A \rfloor \big) =\cmeas\alpha(A)=  \infty$. It would follow that there exists a $\alpha$-counting regular subset $A' \subseteq \lfloor A \rfloor$. The set
\[A \cap \bigcup_{a' \in A'} C_{a'},\]
where $C_{a'}$ is the unit cube based at $a'$, has $\alpha$-counting measure equal to that of $A'$, and so it is an $\alpha$-counting regular subset of $A$.

So, suppose that $A \subseteq \Z^d$ and $\cmeas\alpha(A) = \infty$. If $\alpha = 0$, any non-empty, finite subset $A'$ of $A$ will suffice, so suppose $\alpha > 0$. By Lemma \ref{lem:disjointcubes}, there exists a sequence of pairwise disjoint cubes $(C_n)_n$, $\|C_n\| \geq 1$, $\|C_n\| \rightarrow \infty$, for which
\begin{align}\label{eqn:toinfinity}
\lim_{n \rightarrow \infty} \frac{|A \cap C_n|}{\|C_n\|^\alpha} = \infty.
\end{align}
Let $A_n = A \cap C_n$. In what follows, whenever we pass to a subsequence $(C_{n_k})_k$ of $(C_n)_n$, we replace $A$ with the subset $\bigcup_k A_{n_k}$. Since the cubes $(C_n)_n$ are pairwise disjoint and $\|C_n\| \rightarrow \infty$, we may assume by passing to a subsequence that for all $n \in \N$,
\begin{gather}
\label{eqn:large} \frac{|A \cap C_n|}{\|C_n\|^\alpha} \geq (6n)^d,\\
\label{eqn:distant} C_{n+1} \cap \bigcup_{i=1}^n [C_i]_{2^{n+1}\|C_i\|} = \emptyset,
\end{gather}
where $[C_i]_{2^{n+1}\|C_i\|}$ denotes the cube with the same center as $C_i$ and with side length $\|C_i\| + 2^{n+1}\|C_i\|$.

Property $(\ref{eqn:distant})$ is possible by the fact that any sequence of pairwise disjoint cubes with side lengths bounded from below will eventually be disjoint from some fixed cube, and it means that if $C \subseteq \R^d$ is a cube which intersects both $C_n$ and $C_{n'}$, $1 \leq n < n'$, then \begin{align}\label{eqn:small}\|C\| \geq 2^n\| C_{n}\|.\end{align}


Each finite set $A_n \subseteq C_n$ satisfies the conditions of Lemma \ref{lem:techlemredo} with $\sqrt[d]S / 6 \geq n$. Let $A_n' \subseteq A_n$ and $C_n' \supseteq A_n'$ be the subset and cube guaranteed by the lemma, and let $A' = \bigcup_n A_n' \subseteq A$. In order to show that $A'$ is $\alpha$-counting regular, we will show
\begin{align}\label{eqn:finalgoal}2^d \leq \cmeas\alpha(A') \leq \frac{2^\alpha}{2^\alpha - 1} (2^d+1).\end{align}

By (\ref{eqn:large}) and the Lemma \ref{lem:techlemredo}, $\|C_n'\| \rightarrow \infty$ and \[2^d \leq \frac{|A_n' \cap C_n'|}{\|C_n'\|^\alpha} \leq \frac{|A' \cap C_n'|}{\|C_n'\|^\alpha}.\] This sequence of cubes shows that $2^d \leq \cmeas\alpha(A')$, which is the left hand side of (\ref{eqn:finalgoal}).

To show the right hand side of (\ref{eqn:finalgoal}), it suffices to show for an arbitrary cube $C$, $\| C\| \geq 1$, that
\[\frac{|A' \cap C|}{\| C\|^\alpha} < \frac{2^\alpha}{2^\alpha - 1}(2^d+1).\]
Consider three cases. Case 1: the cube $C$ intersects none of the cubes $\{C_n\}_n$. In this case, $|A' \cap C| = 0$. Case 2: the cube $C$ intersects exactly one of the cubes $\{C_n\}_n$, say $C \cap C_n \neq \emptyset$. In this case, by the choice of $C_n'$,
\[\frac{|A' \cap C|}{\|C\|^\alpha} \leq \frac{|A' \cap C_n'|}{\|C_n'\|^\alpha} < 2^d+1.\]
Case 3: the cube $C$ intersects exactly the cubes $\{C_{i_1}, \cdots, C_{i_m}\}$, $m \geq 2$, $1 \leq i_1 < \cdots < i_m$. Using (\ref{eqn:small}),
\begin{align*}
\frac{|A' \cap C|}{\|C\|^\alpha} \leq \sum_{j=1}^m \frac{ \big|A_{i_j}' \cap C \big|}{\|C\|^\alpha} &\leq \sum_{j=1}^{m-1} \frac{\big|A_{i_j}' \cap C \big|}{(2^{i_j}\|C_{i_j}\|)^\alpha} + \frac{\big|A_{i_m}' \cap C \big|}{\|C\|^\alpha}\\
&< \sum_{j=1}^{m-1} 2^{-\alpha i_j} (2^d+1) + (2^d+1)\\
&< \left(\frac{1}{1-2^{-\alpha}}-1 \right)(2^d+1) + (2^d + 1) = \frac{2^\alpha}{2^\alpha-1}(2^d+1),
\end{align*}
where the third inequality follows from the upper bound in Lemma \ref{lem:techlemredo}.
\end{proof}

The analogue to Proposition \ref{prop:regularsubset} for the mass dimension is stronger, and the proof is simpler. Since we don't need this fact specifically, we leave the proof to the interested reader.

\begin{prop}
Let $A \subseteq \R^d$. For all $0 \leq \alpha \leq \mdim(A)$ and all $0 \leq J \leq \mmeas\alpha(A)$, there exists a subset $A' \subseteq A$ with $\mdim(A') = \alpha$ and $\mmeas\alpha(A') = J$.
\end{prop}

In applications of the Marstrand-type theorems to come, it will be necessary to consider sets which are simultaneously counting and mass regular. It is not the case, however, that all sets contain subsets which are simultaneously counting and mass regular; that is, the straightforward combination of the previous two propositions fails, as the next example indicates.

\begin{example}\label{ex:nosimultsubset}
Let $0 < \alpha < 1$. There exists a set $A \subseteq \Z$ for which
\begin{enumerate}[i.]
\item \label{eqn:propone} $\cdim(A) = \mdim(A) = \alpha$,
\item \label{eqn:proptwo} $\cmeas{\alpha}(A) = \mmeas{\alpha}(A) = \infty$, and
\item \label{eqn:propthree} if $A' \subseteq A$ is such that $\cmeas{\alpha}(A') < \infty$, then $\mmeas{\alpha}(A') = 0$.
\end{enumerate}

Such a set may be constructed as a union of finite sets which are sufficiently spaced. Let $(k_n)_n \subseteq \N$ be rapidly increasing, and set $\eps_n = \left( \log k_n \right)^{-1}$. For each $n \in \N$, let $A_n$ be $N_n = \lfloor k_n 2^{\alpha k_n} \rfloor$ points evenly spaced over the entire interval $[2^{k_n},2^{k_n} + \ell_n]$, where $\ell_n = 2^{\alpha k_n / (\alpha+\eps_n)}$. It is straightforward to check that for all intervals $C \subseteq \R$ with $\|C\| \geq 1$, $| \lfloor A_n \rfloor \cap C| \leq k_n \|C\|^{\alpha + \eps_n}$. Set $A = \cup_n \lfloor A_n \rfloor$.

To show \emph{\ref{eqn:propone}.} and \emph{\ref{eqn:proptwo}.}, it suffices to show that $\cdim(A) \leq \alpha$ and $\mmeas\alpha(A) = \infty$. Note that $N_n \to \infty$ while the density of points in each interval $N_n / \ell_n$ tends to 0 monotonically as $n \to \infty$. For an interval $C \subseteq \R$, let $n$ be such that $\ell_n \leq \|C\| < \ell_{n+1}$. Then either $|A \cap C| \leq N_n$ or $|A \cap C| \leq \|C\| N_{n+1} / \ell_{n+1}$. In either case, one can show that
\[\frac{\log |A \cap C|}{\log \|C\|} \leq \alpha + \eps_n \to \alpha \quad \text{as} \quad \|C\| \to \infty.\]
This implies that $\cdim(A) \leq \alpha$. To show that $\mmeas\alpha(A) = \infty$, it is enough to observe that there are at least $N_n$ elements of $A$ in the interval $\big[-2^{k_n+1},2^{k_n+1}\big)$.

Finally, to show \emph{\ref{eqn:propthree}.}, suppose $A' \subseteq A$ is such that $\cmeas{\alpha}(A') = J < \infty$. Then
\[\big| A' \cap [2^{k_n},2^{k_n} + \ell_n] \big| \leq J \ell_n^\alpha,\]
and since $(k_n)_n$ is increasing rapidly,
\[\frac{\big| A' \cap [-2^{k_n+1},2^{k_n+1}) \big|}{2^{(k_n+2)\alpha}} \leq \frac{J \ell_n^\alpha + 1}{2^{(k_n+2)\alpha}} \to 0 \quad \text{as} \quad n \to \infty.\]
This implies that $\mmeas\alpha(A') = 0$.
\end{example}

\subsection{The counting dimension via coverings}

The goal of this section is to provide a characterization of the counting dimension via coverings.

\begin{definition}
Let $A \subseteq \R^d$ and $\alpha \geq 0$. The \emph{$\alpha$-covering measure} of $A$ is
\[\covmeas{\alpha,d}(A) = \lim_{r \rightarrow 0} \limsup_{\|C\| \rightarrow \infty} \inf \left\{ \ \sum_i \left( \frac{\|C_i\|}{\|C\|} \right)^\alpha \ \middle| \ A \cap C \subseteq \bigcup_i C_i, \ 1 \leq \|C_i\| \leq r \|C\| \right\},\]
where the limit supremum is taken over cubes $C \subseteq \R^d$, the infimum is taken over all sets of cubes $\{C_i\}$ satisfying the given conditions, and the infimum of the empty set is taken to be 0. (To save space, the constraints on the covers $\{C_i\}$ over which the infimum is taken will often be omitted.) The \emph{covering dimension} of $A$ is
\[\covdim^d(A) = \inf \{\alpha \geq 0 \ | \ \covmeas\alpha(A) = 0 \}.\]
As before, when the ambient space is apparent, we write $\covmeas\alpha$ and $\covdim$ instead of $\covmeas{\alpha,d}$ and $\covdim^d$.
\end{definition}

The letter H was chosen to denote the covering measures due to the similarity in formulation with the classical Hausdorff measures; note that in keeping with the course geometry perspective, there is a lower bound on the side length of the cubes $C_i$ in admissible covers $\{C_i\}$.

\begin{lemma}\label{lem:basicprops}
Let $A \subseteq \R^d$ and $\alpha \geq 0$. Then
\begin{enumerate}[i.]
\item for all $\alpha' \leq \alpha$, $\covmeas{\alpha'}(A) \geq \covmeas\alpha (A)$;
\item for all $A' \subseteq A$, $\covmeas\alpha(A') \leq \covmeas\alpha(A)$ and $\covdim(A') \leq \covdim(A)$;
\item $0 \leq \covdim(A) \leq d$;
\item if $\alpha < \covdim(A)$, then $\covmeas\alpha(A) = \infty$;
\item if $\alpha > \covdim(A)$, then $\covmeas\alpha(A) = 0$.
\end{enumerate}
\end{lemma}

\begin{proof}
Properties \emph{i.} and \emph{ii.} are immediate from the definition of the covering measures. Property \emph{iii.} follows from the fact that if $\alpha > d$, then $\covmeas\alpha (A) = 0$. To prove this, for all $0 < r < 1$ and all cubes $C$, take a covering $\{C_i\}_i$ of $A \cap C$ with $\| C_i \| = r \|C\|$. Then
\[\inf_{\{C_i\}} \left\{ \ \sum_i \left( \frac{\|C_i\|}{\|C\|} \right)^\alpha \right\} \leq \left\lceil \frac{1}{r} \right\rceil^d r^\alpha \leq 2^d r^{\alpha - d}.\]
Since $C$ was arbitrary,
\[\limsup_{\|C\| \rightarrow \infty} \inf_{\{C_i\}} \left\{ \ \sum_i \left( \frac{\|C_i\|}{\|C\|} \right)^\alpha \right\} \leq 2^d r^{\alpha - d}.\]
Since $\alpha - d > 0$, $r^{\alpha - d} \rightarrow 0$ as $r \rightarrow 0$, and so $\covmeas\alpha(A) = 0$.

To prove \emph{iv.}, it suffices by the monotonicity of $\covmeas\alpha(A)$ in $\alpha$ (property \emph{i.}) to show that if $\alpha$ is such that $\covmeas{\alpha}(A) > 0$, then for all $\alpha' < \alpha$, $\covmeas{\alpha'}(A) = \infty$. So, suppose $\covmeas{\alpha}(A) > 0$. It follows that there exists an $\epsilon > 0$ and an $R > 0$ such that for all $0 < r < R$,
\[\limsup_{\|C\| \rightarrow \infty} \inf_{\{C_i\}} \left\{ \ \sum_i \left( \frac{\|C_i\|}{\|C\|} \right)^{\alpha} \right\} > \eps.\]
Fix $0 < r < R$, and let $(C_n)_n$, $\|C_n\| \rightarrow \infty$, be a sequence of cubes such that for all $n \in \N$,
\[\inf_{\{C_{n,i}\}} \left\{ \ \sum_i \left( \frac{\|C_{n,i}\|}{\|C_n\|} \right)^{\alpha} \right\} > \eps.\]
Let $n \in \N$ and $\{C_{n,i}\}_i$ be a cover of $A \cap C_n$ with $1 \leq \|C_{n,i}\| \leq r \|C_n\|$. Then
\begin{align*}
\sum_i \left( \frac{\|C_{n,i}\|}{\|C_n\|} \right)^{\alpha'} &= \sum_i \left( \frac{\|C_n\|}{\|C_{n,i}\|} \right)^{\alpha - \alpha'} \left( \frac{\|C_{n,i}\|}{\|C_n\|} \right)^{\alpha}\\
&\geq \left( \frac{1}{r} \right)^{\alpha - \alpha'} \sum_i \left( \frac{\|C_{n,i}\|}{\|C_n\|} \right)^{\alpha} > \left( \frac{1}{r} \right)^{\alpha - \alpha'} \eps.
\end{align*}
Since $n \in \N$ was arbitrary,
\[\limsup_{n \rightarrow \infty} \inf_{\{C_{n,i}\}} \left\{ \ \sum_i \left( \frac{\|C_{n,i}\|}{\|C_n\|} \right)^{\alpha'} \right\} > \left( \frac{1}{r} \right)^{\alpha - \alpha'} \eps,\]
and, consequently,
\[\limsup_{\|C\| \rightarrow \infty} \inf_{\{C_i\}} \left\{ \ \sum_i \left( \frac{\|C_i\|}{\|C\|} \right)^{\alpha'} \right\} > \left( \frac{1}{r} \right)^{\alpha - \alpha'} \eps.\]
Since $\alpha - \alpha' > 0$ and $0 < r < R$ was arbitrary, $\covmeas{\alpha'} (A) = \infty$.

Finally, property \emph{v.} follows from the definition of $\covdim$ and the monotonicity of $\covmeas\alpha(A)$ in $\alpha$.
\end{proof}

The analogue of Lemma \ref{lem:countingprops} holds for the covering measures and dimension. Instead of proving it separately, we turn to proving the equivalence of the covering dimension and the counting dimension.

\begin{thm}\label{thm:dimsareequiv}
The covering dimension and the counting dimension coincide; that is, for all $A \subseteq \R^d$, \[\covdim(A) = \cdim(A).\]
\end{thm}

\begin{proof}
Let $A \subseteq \R^d$. Fix $r > 0$, and let $C \subseteq \R^d$ be a cube with $r \|C\| \geq 1$. By covering $A \cap C$ with unit cubes, we see
\[\inf \left\{ \sum_i \left(\frac{\|C_i\|}{\|C\|} \right)^\alpha \ \middle| \ A \cap C \subseteq \bigcup_i C_i, \ 1 \leq \|C_i\| \leq r \|C\| \right\} \leq \frac{\big| \lfloor A \rfloor \cap C \big|}{\|C\|^\alpha}.\]
Since $r$ and $C$ were arbitrary, it follows that
\[\covmeas\alpha(A) = \lim_{r \rightarrow 0} \limsup_{\|C\| \rightarrow \infty} \inf_{\{C_i\}} \left\{ \sum_i \left(\frac{\|C_i\|}{\|C\|} \right)^\alpha \right\} \leq \limsup_{\|C\| \rightarrow \infty} \frac{\big| \lfloor A \rfloor \cap C \big|}{\|C\|^\alpha} = \cmeas\alpha(A),\]
whereby $\covdim(A) \leq \cdim(A)$.

We claim that it suffices to have the reverse inequality $\cdim(A') \leq \covdim(A')$ for counting regular subsets $A' \subseteq \R^d$ to prove it for $A$. Indeed, if $\cdim(A) = 0$, then $\cdim(A) \leq \covdim(A)$ and we are done. Otherwise, by Proposition \ref{prop:regularsubset}, for all $0 \leq \alpha < \cdim(A)$, there exists an $\alpha$-counting regular subset $A' \subseteq A$. By the monotonicity of $\covdim$,
\[\alpha = \cdim(A') \leq \covdim(A') \leq \covdim(A).\]
Since $\alpha < \cdim(A)$ was arbitrary, $\cdim(A) \leq \covdim(A)$.

So, assume $A$ is counting regular; we want to show $\cdim(A) \leq \covdim(A)$. Let $\alpha = \cdim(A)$. Since $\cmeas\alpha(A) < \infty$, there exists a $0 < K < \infty$ with the property that for all cubes $C \subseteq \R^d$ with $\|C\| \geq 1$, $\big| \lfloor A \rfloor \cap C \big| < K\|C\|^\alpha$.  For any cube $C$, $\|C\| \geq 1$, and any cover $\{C_i\}_i$ of $A \cap C$ with $\|C_i\| \geq 1$,
\[\sum_i \left(\frac{\|C_i\|}{\|C\|} \right)^\alpha \geq \frac{1}{\|C\|^\alpha}\sum_i \frac{\big| \lfloor A \rfloor \cap C_i \big|}{K} \geq \frac{\big| \lfloor A \rfloor \cap C\big|}{K\|C\|^\alpha}.\]
Therefore, for any $r>0$,
\[\inf_{\{C_i\}} \left\{ \sum_i \left(\frac{\|C_i\|}{\|C\|} \right)^\alpha \right\} \geq \frac{\big| \lfloor A \rfloor \cap C \big|}{K\|C\|^\alpha}.\]
Since $C$ was arbitrary,
\[\limsup_{\|C\| \rightarrow \infty} \inf_{\{C_i\}} \left\{ \sum_i \left(\frac{\|C_i\|}{\|C\|} \right)^\alpha \right\} \geq \limsup_{\|C\| \rightarrow \infty} \frac{\big| \lfloor A \rfloor \cap C \big|}{K\|C\|^\alpha} = \frac{\cmeas\alpha (A)}{K}.\]
Since $r>0$ was arbitrary and $\cmeas\alpha(A) > 0$,
\[\covmeas\alpha(A) \geq \frac{\cmeas\alpha(A)}{K} > 0,\]
meaning that $\covdim(A) \geq \alpha = \cdim(A)$.
\end{proof}

It follows from the proof that if $A \subseteq \R^d$ is $\alpha$-counting regular, then
\[0< \covmeas\alpha (A) \leq \cmeas\alpha(A) < \infty.\]
Along with Proposition \ref{prop:regularsubset}, this provides a (partial) analogue of the same proposition for the covering measures: for all $\alpha < \covdim(A) = \cdim(A)$, there exists a subset $A' \subseteq A$ for which $0 < \covmeas\alpha(A') < \infty$. There are examples of sets $A \subseteq \R$ for which $\covmeas\alpha(A)$ is finite while $\cmeas\alpha(A)$ is infinite.

Finally, Barlow and Taylor \cite{bat2} consider a version of the covering dimension which they call the discrete Hausdorff dimension. If the limit supremum in the definition of $\covmeas\alpha$ is taken over centered cubes, the resulting dimensional quantity may be shown to be equal to $\dim_L$ (in Barlow and Taylor's notation); in particular, one does not recover the upper mass dimension.

\subsection{Invariance under quasi-isometric embeddings}\label{sec:invunderqie}

The goal of this section is to show that the counting and mass dimensions are invariant under maps which are Lipschitz up to an additive constant. By definition, rounding any set in $\R^d$ to the integer lattice $\Z^d$ does not affect its regularity or dimension; more generally, the same is true for rounding to any full rank sublattice. What is more, it is easy to check that $\cdim(A) = \cdim(c A + z)$ for all $c > 0$ and $z \in \R^d$. These examples are special cases of the fact that both dimensions are invariant under quasi-isometric embeddings.

The asymptotic notation used below is standard. Given two functions $f$ and $g$, we write $f \ll_{a,b,\ldots} g$ or $g \gg_{a,b,\ldots} f$ if there exists a constant $K > 0$ depending at most on the quantities $a, b, \ldots$ for which $f(x) \leq K g(x)$ for all $x$ in the common domain of $f$ and $g$ (unless another domain is specified). We write $f \asymp g$ if both $f \ll g$ and $g \ll f$.

\begin{definition}[\cite{delaharpe}]\label{def:qie}
Let $(X,d_X)$ and $(Y,d_Y)$ be metric spaces. A map $f: X \to Y$ is a \emph{quasi-isometric embedding} if there exist constants $K \geq 1$ and $M \geq 0$ such that for all $x_1, x_2 \in X$,
\[\frac{1}{K}d_X \big(x_1,x_2 \big) - M \leq d_Y \big(f(x_1),f(x_2) \big) \leq Kd_X \big(x_1,x_2 \big) + M.\]
Note that $f$ need not be injective.
\end{definition}

The main results are Propositions \ref{prop:quasiisometryinvariance} and \ref{prop:quasiisometryinvarianceofmass} giving that the counting and mass measures of a set and those of its image under quasi-isometric embeddings are equivalent. After proving Proposition \ref{prop:quasiisometryinvariance}, we derive the corollaries necessary later in this work. We will not need Proposition \ref{prop:quasiisometryinvarianceofmass} specifically, but we provide a proof of it for completeness.

\begin{lemma}\label{lem:quasilemma}
For all bounded $B \subseteq \R^k$ and all $(K,M)$-quasi-isometric embeddings $f: B \to \R^d$,
\[\big| \lfloor B \rfloor \big| \ll_{k,d,K,M} \big| \lfloor f(B) \rfloor \big|.\]
\end{lemma}

\begin{proof}
It suffices to show that if $B' \subseteq B$ is such that $f(B')$ is contained in a unit cube, then
\[\big| \lfloor B' \rfloor \big| \ll_{k,d,K,M} 1.\]
Indeed, the conclusion in the lemma follows immediately by considering the partition of $f(B)$ induced by $\Z^d$ and summing.

It follows from the definition of a $(K,M)$-quasi-isometric embedding and the assumption that $f(B')$ is contained in a unit cube that
\[\frac{1}{K} \diam (B') - M \leq \diam \big(f(B')\big) \ll_{d} 1.\]
Therefore, $\diam (B') \ll_{d,K,M} 1$. For all finite sets $A \subseteq \Z^{k}$, $|A| \ll_{k} \big(\diam(A) + 1 \big)^{k}$, and so
\[\big| \lfloor B' \rfloor \big| \ll_{k} \big(\diam(\lfloor B' \rfloor) + 1\big)^{k} \leq \big(\diam( B') + 2\big)^{k} \ll_{k,d,K,M} 1.\]
\end{proof}

\begin{prop}\label{prop:quasiisometryinvariance}
For all $A \subseteq \R^k$ and all $(K,M)$-quasi-isometric embeddings $f: A \to \R^d$,
\[\cmeas{\alpha,k}(A) \asymp_{\alpha,k,d,K,M} \cmeas{\alpha,d}\big(f(A)\big).\]
In particular, $A$ is $\alpha$-counting regular if and only if $f(A)$ is $\alpha$-counting regular, and $\cdim(A) = \cdim \big(f(A) \big)$.
\end{prop}

\begin{proof}
(In what follows, dependence on $\alpha, k, d, K,$ and $M$ in the asymptotic constants will be suppressed.) It suffices to show that
\begin{align}\label{eqn:quasionedirection}\cmeas{\alpha,k}(A) \ll \cmeas{\alpha,d}\big(f(A)\big).\end{align}
To see why, note that if $f: A \to \R^d$ is a $(K,M)$-quasi-isometric embedding, then any right inverse $g: f(A) \to A$ is a $(K,KM)$-quasi-isometric embedding. Therefore, by (\ref{eqn:quasionedirection}) for $g$ and the fact that $g(f(A)) \subseteq A$,
\[\cmeas{\alpha,d}\big(f(A)\big) \ll \cmeas{\alpha,k} \big(g\big(f(A)\big)\big) \leq \cmeas{\alpha,k} (A).\]

Let $C \subseteq \R^{k}$ be a cube, $\|C\| \geq 1$. Denote by $[C]_2$ the cube with the same center as $C$ and with side length $\|C\| + 2$. Since $f$ is a quasi-isometric embedding, there exists a cube $C' \subseteq \R^{d}$ satisfying
\begin{align}\label{eqn:choiceoftildec}\big\lfloor f(A \cap [C]_2 )\big\rfloor \subseteq C' \qquad \text{and} \qquad \|C'\| \ll \|C\|.\end{align}
Since $\lfloor A \rfloor \cap C \subseteq \big\lfloor A \cap [C]_2 \big\rfloor$, Lemma \ref{lem:quasilemma} and (\ref{eqn:choiceoftildec}) give
\[\big|\lfloor A \rfloor \cap C\big| \leq \big| \lfloor A \cap [C]_2 \rfloor \big| \ll \big| \lfloor f(A \cap [C]_2) \rfloor \big| \leq \big| \lfloor f(A) \rfloor \cap C' \big|.\]
Let $(C_n)_n$, $\|C_n\| \to \infty$, be a sequence of cubes in $\R^{k}$ along which the limit supremum in $\cmeas{\alpha,k}(A)$ is achieved, and, for each $n \in \N$, let $C_n'$ be a cube in $\R^{d}$ satisfying (\ref{eqn:choiceoftildec}) for $C_n$. By the remarks above,
\[\cmeas{\alpha,k}(A) = \lim_{n \to \infty} \frac{\big|\lfloor A \rfloor \cap C_n \big|}{\|C_n\|^\alpha} \ll \limsup_{n \to \infty} \frac{\big|\lfloor f(A) \rfloor \cap C_n'\big|}{\|C_n'\|^\alpha} \leq \cmeas{\alpha,d}\big(f(A)\big).\]
\end{proof}

Note that the asymptotic constants appearing in the conclusion of the proposition are independent of both the set $A$ and the embedding $f$ (as long as it is $(K,M)$-quasi-isometric).

We will now develop two corollaries which will be used several times throughout this work. For the counting measures, the corollaries will follow immediately from Proposition \ref{prop:quasiisometryinvariance}. For the mass measures, the map $f$ (and a right inverse for it) in both corollaries will be such that $|f(a)|$ is controlled by $|a|$; this will allow us to prove the desired results by modifying the proof of Proposition \ref{prop:quasiisometryinvariance} by taking $C$ and $C'$ to be centered cubes in (\ref{eqn:choiceoftildec}).

\begin{definition}
The Hausdorff distance between subsets $A, B \subseteq \R^d$ is
\[\hdist (A,B) = \inf \big\{ \delta \geq 0 \ \big| \ A \subseteq B_\delta \text{ and } B \subseteq A_\delta \big\},\]
where $A_\delta$ denotes the closed $\delta$-neighborhood of $A$ in $\R^d$.
\end{definition}

\begin{cor}\label{cor:qiecorone}
Let $A, B \subseteq \R^d$. If $\hdist(A,B) \leq \ell < \infty$, then
\[\cmeas\alpha(A) \asymp_{\alpha,d,\ell} \cmeas\alpha(B).\]
In particular, $A$ is $\alpha$-counting regular if and only if $B$ is $\alpha$-counting regular, and $\cdim(A) = \cdim(B)$.

The same conclusions hold for the mass measures and dimension.
\end{cor}

\begin{proof}
If $\hdist(A,B) \leq \ell$, then there exist $(1,2\ell)$-quasi-isometric embeddings $A \to B$ and $B \to A$. By monotonicity of the counting measures, this suffices to give the result for the counting measures and dimension immediately by Proposition \ref{prop:quasiisometryinvariance}.

For the analogous result for the mass measures and dimension, consider the proof of Proposition \ref{prop:quasiisometryinvariance}. Since points move by at most a distance $\ell$ under both quasi-isometric embeddings, if $C$ is a centered cube, the cube $C'$ in (\ref{eqn:choiceoftildec}) may be chosen to be centered and satisfy $\|C'\| \leq \|C\| + 2 \ell$. By taking the cubes $C_n$ to be centered, the result for the mass measures follows.
\end{proof}

This corollary will allow us the simplifying step of passing from $A$ to $\lfloor A \rfloor$ or from $TA$ to $T \lfloor A \rfloor$ (where $T$ will be linear) when determining the counting and mass regularity and dimension of these sets. Also note that by considering the shift map $A \to A + s$, this corollary shows that the mass regularity and dimension of $A$ is invariant under changes to the ``base point'' with respect to which the mass measures are defined.

\begin{cor}\label{cor:qiecortwo}
Let $R \subseteq \R^k$ be a linear subspace and $T: R \to \R^d$ be an invertible linear transformation. Set $K = \max \big( \|T\|, \|T^{-1}\| \big)$, where $\|T\|$ denotes the operator norm of $T$. For all $A \subseteq R$,
\[\cmeas{\alpha,k}(A) \asymp_{\alpha,k,d,K} \cmeas{\alpha,d}(TA),\]
 In particular, $A$ is $\alpha$-counting regular if and only if $TA$ is $\alpha$-counting regular, and $\cdim(A) = \cdim(TA)$.

The same conclusions hold for the mass measures and dimension.
\end{cor}

\begin{proof}
Note that $T$ and $T^{-1}$ are $(K,0)$-quasi-isometric embeddings. This suffices to give the result for the counting measures and dimension immediately by Proposition \ref{prop:quasiisometryinvariance}.

For the analogous result for the mass measures and dimension, note that since $T$ is linear, if $C$ is a centered cube, the cube $C'$ in (\ref{eqn:choiceoftildec}) may be chosen to be centered and satisfy $\|C'\| \leq K\|C\|$. The same holds for $T^{-1}$, and the result for the mass measures follows.
\end{proof}

Finally, for completeness, we include a proof of the fact that the mass dimension is invariant under general quasi-isometric embeddings. The statement is slightly weaker than the analogous statement for the counting measures since there is now dependence on where the map sends points with respect to the origin.

\begin{prop}\label{prop:quasiisometryinvarianceofmass}
Let $X \subseteq \R^k$ and $f: X \to \R^d$ be a quasi-isometric embedding. For all subsets $A \subseteq X$,
\[\mmeas{\alpha,k}(A) \asymp_{\alpha,f} \mmeas{\alpha,d}\big(f(A)\big).\]
In particular, $A$ is $\alpha$-mass regular if and only if $f(A)$ is $\alpha$-mass regular, and $\mdim(A) = \mdim(f(A))$.
\end{prop}

\begin{proof}
The proof follows the lines of the proof of Proposition \ref{prop:quasiisometryinvariance}, but extra care must be taken to establish the analogue to (\ref{eqn:choiceoftildec}). As in the proof of Proposition \ref{prop:quasiisometryinvariance}, it suffices to show that for all $A \subseteq X$,
\[\mmeas{\alpha,k}(A) \ll_{\alpha,f} \mmeas{\alpha,d}\big(f(A)\big).\]
Suppose $f$ is a $(K,M)$-quasi-isometric embedding, and fix $x_0 \in X$. For all $x \in X$,
\[\big|f(x)\big| \leq \big|f(x) - f(x_0)\big| + \big|f(x_0)\big| \leq K|x| + K|x_0| + M + \big|f(x_0)\big|.\]
Since $x_0 \in X$ was arbitrary, for all $x \in X$,
\begin{align}\label{eqn:importantbound} \big|f(x)\big| \leq K |x| + M + \inf_{x_0 \in X} \big(K|x_0| + \big|f(x_0)\big| \big) \ll_f |x| + 1.\end{align}
Let $A \subseteq X$. Let $C \subseteq \R^k$ be a centered cube, $\|C\| \geq 1$, and let $[C]_2$ be as in the proof of Proposition \ref{prop:quasiisometryinvariance}. By (\ref{eqn:importantbound}), there exists a centered cube $C' \subseteq \R^d$ satisfying
\begin{align*}\big\lfloor f(A \cap [C]_2 ) \big\rfloor \subseteq C' \qquad \text{and} \qquad \|C'\| \ll_f \|C\|.\end{align*}
The rest of the proof follows exactly as in the proof of Proposition \ref{prop:quasiisometryinvariance}.
\end{proof}

\subsection{Product sets, compatibility, and universality}\label{sec:products}

Understanding the behavior of the counting and mass measures of Cartesian products of sets is important in applications of the Marstrand-type theorems to come. The terminology and main definitions in this section follow those of Lima and Moreira \cite{lima}.

It will be useful throughout this section to write cubes $C \subseteq \R^{d_1} \times \cdots \times \R^{d_m}$ as a product of cubes $C = C_1 \times \cdots \times C_m$,
where each $C_i \subseteq \R^{d_i}$ and $\|C_i\| = \|C\|$. In the other direction, given cubes $C_i \subseteq \R^{d_i}$, $1 \leq i \leq m$, with equal side lengths, the set $C_1 \times \cdots \cdots \times C_m$ is a cube in $\R^{d_1} \times \cdots \times \R^{d_m}$. We will use this notation consistently in this section without further mention.

\begin{lemma}\label{lem:measofprodsets}
For each $1 \leq i \leq m$, let $A_i \subseteq \R^{d_i}$ be non-empty and $\alpha_i \geq 0$, and set $\alpha = \sum_i \alpha_i$. Then
\[\max_{1\leq i \leq m} \cmeas\alpha (A_i) \leq \cmeas{\alpha} \big( A_1 \times \cdots \times A_m \big) \leq \prod_{i=1}^m \cmeas{\alpha_i}(A_i),\]
where the product on the right hand side is taken to be infinity if any one of the terms is infinity.

The same conclusions hold for the mass measures.
\end{lemma}

\begin{proof}
Note that
\[\cmeas{\alpha} \big( A_1 \times \cdots \times A_m \big) =\limsup_{\|C\| \rightarrow \infty} \prod_{i = 1}^m \frac{| A_i \cap C_i |}{\|C\|^{\alpha_i}} \leq \prod_{i = 1}^m \cmeas{\alpha_i} (A_i),\]
where the inequality holds if we take the product on the right hand side to be infinity if any one of the terms is infinity. For the lower bound, write
\[\cmeas{\alpha} \big( A_1 \times \cdots \times A_m \big) =\limsup_{\|C\| \rightarrow \infty} \left( \frac{| A_j \cap C_j |}{\|C\|^{\alpha}} \prod_{\substack{i=1 \\ i \neq j}}^m \big| A_i \cap C_i \big| \right) \geq \cmeas\alpha(A_j),\]
where the inequality holds since the limit supremum is attained along a sequence with $A_i \cap C_i$ non-empty.
\end{proof}

\begin{cor}\label{cor:prodsets}
For all non-empty $A_i \subseteq \R^{d_i}$, $1 \leq i \leq m$,
\begin{align}\label{eqn:diminequality}\max_{1 \leq i \leq m} \cdim(A_i) \leq \cdim \left( A_1 \times \cdots \times A_m \right) \leq \sum_{i=1}^m \cdim(A_i).\end{align}

The same conclusions hold for the mass dimension.
\end{cor}

\begin{proof}
For $\alpha_1 < \cdim(A_1)$, we have $\cmeas{\alpha_1}(A_1) = \infty$, and so Lemma \ref{lem:measofprodsets} with $\alpha_i = 0$ for $i \neq 1$ implies that $\cmeas{\alpha_1}(A_1 \times \cdots \times A_m) = \infty$. Letting $\alpha_1$ tend to $\cdim(A_1)$ from below shows that 
\[\cdim(A_1) \leq \cdim(A_1 \times \cdots \times A_m),\]
and the lower bound in (\ref{eqn:diminequality}) follows by the same considerations with the other sets. For the upper bound, for each $1 \leq i \leq m$, let $\alpha_i$ tend to $\cdim(A_i)$ from above and use in the same way Lemma \ref{lem:basiccountingprops} and the upper bound from Lemma \ref{lem:measofprodsets}.
\end{proof}


The remainder of this section is focused on tools which will aid in applications.  It will be most important to have information on the dimension and regularity of $A_1 \times \cdots \times A_m$ based on the dimension and regularity of the $A_i$'s.

\begin{definition}\label{def:compatible}
A collection of non-empty sets $\{A_1, \ldots, A_m \}$, $A_i \subseteq \R^{d_i}$, is \emph{counting compatible} if
\[\cdim \big( A_1 \times \cdots \times A_m \big) = \sum_{i=1}^m \cdim(A_i).\]
The collection $\{A_1, \ldots, A_m \}$ is \emph{strongly counting compatible} if each $A_i$ is $\alpha_i$-counting regular and $A_1 \times \cdots \times A_m$ is $\left( \sum_{i=1}^m \alpha_i \right)$-counting regular.  Two sets $A$ and $B$ are (strongly) counting compatible if the collection $\{A,B\}$ is (strongly) counting compatible.

\emph{Mass compatibility} and \emph{strong mass compatibility} are defined analogously.
\end{definition}

It follows immediately from the definitions that strong compatibility implies compatibility. We now proceed to give necessary and sufficient conditions for these two properties.

\begin{lemma}\label{lem:compatibility}
For $1 \leq i \leq m$, let $A_i \subseteq \R^{d_i}$ be non-empty. The collection $\{A_1, \ldots, A_m\}$ is counting compatible if and only if there exists a sequence $(\ell_n)_n \subseteq \N$, $\ell_n \rightarrow \infty$, such that for all $1 \leq i \leq m$, there exists a sequence $(C_{n,i})_n$ of cubes in $\R^{d_i}$ with $\|C_{n,i}\| = \ell_n$ along which $A_i$ achieves its counting dimension (recall Definition \ref{def:achievesdimension}).

The same conclusions hold for mass compatibility with ``counting dimension'' replaced by ``mass dimension'' and where the cubes $C_{n,i}$ are centered.
\end{lemma}

\begin{proof}
For the ``if'' direction, for each $n \in \N$, consider the cube $C_n = C_{n,1} \times \cdots \times C_{n,m}$ in $\R^{d_1} \times \cdots \times \R^{d_m}$ of side length $\ell_n$. By (\ref{eqn:explicitformula}) and the assumption that $A_i$ achieves its counting dimension along $(C_{n,i})_n$,
\[\cdim(A_1 \times \cdots \times A_m) \geq \lim_{n \to \infty} \sum_{i =1}^m \frac{\log |A_i \cap C_{n,i} |}{\log \|C_{n}\|} = \sum_{i = 1}^m \cdim(A_i).\]
The reverse inequality follows by Corollary \ref{cor:prodsets}.

For the ``only if'' direction, assuming the collection is counting compatible, there exists a sequence of cubes $(C_n)_n$ in $\R^{d_1} \times \cdots \times \R^{d_m}$, $\|C_n\| \rightarrow \infty$, which satisfies
\[\lim_{n \rightarrow \infty} \sum_{i=1}^m \frac{\log \left| A_i \cap C_{n,i} \right|}{\log \|C_n\|} = \sum_{i=1}^m \cdim(A_i).\]
Since each $\|C_n\| = \|C_{n,i}\|$ and $\limsup_{n \rightarrow \infty} \log \left| A_i \cap C_{n,i} \right| \big / \log \|C_n\| \leq \cdim(A_i)$, we have
\[\lim_{n \rightarrow \infty} \frac{\log \left| A_i \cap C_{n,i} \right|}{\log \|C_n\|} = \lim_{n \rightarrow \infty} \frac{\log \left| A_i \cap C_{n,i} \right|}{\log \|C_{n,i}\|} = \cdim(A_i).\]
Now the sequence $(\ell_n = \|C_n\|)_n$ and, for each $1 \leq i \leq m$, the sequences $(C_{n,i})_n$ are such that $\|C_{n,i}\| = \ell_n$ and $A_i$ achieves its counting dimension along $(C_{n,i})_n$.
\end{proof}

The previous lemma says that two sets are compatible if and only if their dimensions are computable along cubes of equal side lengths. At the other extreme, equality on the left hand side in (\ref{eqn:diminequality}) is possible if the sets are ``totally mutually incompatible.'' It is possible, for example, to have $\cdim(A_1) = \cdim(A_2) = 1$ and $\cdim(A_1 \times A_2) = 1$.

\begin{lemma}\label{lem:strongcompatibility}
For $1 \leq i \leq m$, let $A_i \subseteq \R^{d_i}$ be $\alpha_i$-counting regular. The collection $\{A_1, \ldots, A_m\}$ is strongly counting compatible if and only if there exists a sequence $(\ell_n)_n \subseteq \N$, $\ell_n \rightarrow \infty$, such that for all $1 \leq i \leq m$, there exists a sequence $(C_{n,i})_n$ of cubes in $\R^{d_i}$ with $\|C_{n,i}\| = \ell_n$ for which
\[\lim_{n \to \infty} \prod_{i = 1}^m \frac{| A_i \cap C_{n,i} |}{\|C_{n,i}\|^{\alpha_i}} > 0.\]

The same conclusions hold for strong mass compatibility where the cubes $C_{n,i}$ are centered.
\end{lemma}

\begin{proof}
Let $\alpha = \sum_{i=1}^m \alpha_i$, and note that by the upper bound in Lemma \ref{lem:measofprodsets}, $\cmeas\alpha(A_1 \times \cdots \times A_m) < \infty$. Therefore, the strong counting compatibility of the collection is equivalent to the positivity of $\cmeas\alpha(A_1 \times \cdots \times A_m)$.

Note that
\[\cmeas{\alpha} \big( A_1 \times \cdots \times A_m \big) = \sup_{\substack{(\ell_n)_n \\ \ell_n \to \infty}} \sup_{\substack{(C_n)_n \\ \|C_n\| = \ell_n}} \limsup_{n \to \infty} \prod_{i = 1}^m \frac{| A_i \cap C_{n,i} |}{\|C_{n,i}\|^{\alpha_i}},\]
where the first supremum is taken over sequences in $\N$ and the second is taken over sequences of cubes in $\R^{d_1} \times \cdots \times \R^{d_m}$. The conclusion follows immediately from this expression (by passing, if necessary, to a subsequence along which the limit exists).
\end{proof}

\begin{remark}
As a consequence of Lemmas \ref{lem:compatibility} and \ref{lem:strongcompatibility}, we see that a collection is (strongly) counting or mass compatible if and only if all of its subcollections are (strongly) counting or mass compatible.
\end{remark}

Sets which are compatible with all other sets are useful in applications.

\begin{definition}
A set is \emph{(strongly) counting universal} if it is (strongly) counting compatible with all other (regular) sets.  A set is \emph{(strongly) mass universal} if it is (strongly) mass compatible with all other (regular) sets.
\end{definition}

A set is universal if its dimension may be measured along any sequence of cube side lengths. The proof of the following lemma follows directly from Lemmas \ref{lem:compatibility} and \ref{lem:strongcompatibility}.

\begin{lemma}\label{lem:universality}
A set $A \subseteq \R^d$ is counting universal if there exists a sequence of cubes $(C_n)_n \subseteq \R^d$ with $\|C_n\| = n$ along which $A$ achieves its counting dimension. If $A$ is  $\alpha$-counting regular, then it is strongly counting universal if there exists a sequence of cubes $(C_n)_n \subseteq \R^d$ with $\|C_n\| = n$ for which
\[\liminf_{n \to \infty} \frac{| A \cap C_n |}{\|C_n\|^\alpha} > 0.\]

The same conclusions hold for (strong) mass universality with ``counting dimension/regularity'' replaced by ``mass dimension/regularity'' and where the cubes $C_n$ are centered.
\end{lemma}

We conclude this section by giving examples of universal and non-universal sets as well as examples of compatible collections.

\begin{examples}\label{ex:univexamples}\leavevmode
\begin{enumerate}[i.]
\item Let $f: \R \to \R$ be a polynomial of degree $n \geq 1$ and leading coefficient $a_n \neq 0$. It is an exercise to check that
\begin{align*}
\big(2|a_n| \big)^{-1/n} &\leq \liminf_{\ell \to \infty} \frac{\big| \lfloor f(\Z) \rfloor \cap [-\ell,\ell) \big|}{(2\ell)^{1/n}}\\
&\leq \limsup_{\ell \to \infty} \frac{\big| \lfloor f(\Z) \rfloor \cap [-\ell,\ell) \big|}{(2\ell)^{1/n}} \leq 2 \big(2|a_n| \big)^{-1/n},
\end{align*}
whereby the set $f(\Z)$ is strongly mass and strongly counting universal.

\item By the prime number theorem, the set of prime numbers is mass and counting universal (but not strongly so, since the set of primes is neither counting nor mass regular).

\item \label{item:integercantortwo} Let $E_M$ be the integer Cantor set associated to the base $b \in \N$ and binary matrix $M$ described in example \ref{item:integercantor}. in Section \ref{sec:examples}. As shown in \cite{lima}, there exists a $K = K(M) > 0$ such that for all $n \in \N$,
\[K^{-1}\lambda_+(M)^{n} \leq \big|E_M \cap [0,b^{n+1}) \big| \leq K \lambda_+(M)^{n},\]
and $\cdim(E_M) = \log \lambda_+(M) \big/ \log b$. For $\ell \in \N$, if $n$ is such that $b^n \leq \ell < b^{n+1}$, then
\[\frac{1}{K 2^{\cdim(E_M)} \lambda_+(M)^{2}} \leq \frac{\big|E_M \cap [-\ell,\ell) \big|}{(2\ell)^{\cdim(E_M)}} \leq \frac{K}{2^{\cdim(E_M)}}.\]
This shows that $E_M$ is mass regular and that
\[\liminf_{\ell \to \infty} \frac{\big| E_M \cap [-\ell,\ell) \big|}{(2\ell)^{\cdim(E_M)}} > 0,\]
whereby $E_M$ is strongly mass universal by Lemma \ref{lem:universality}. In fact, it is true that $E_M$ is counting regular (see \cite{lima}, Lemma 3.4) and so, by the previous observation, also strongly counting universal.

\item Let $E$ be the generalized IP set corresponding to the sequences $(k_n)_n$ and $(d_n)_n$ of positive integers described in example \ref{item:genipset}. in Section \ref{sec:examples}. As shown in \cite{lima}, $\cdim(E) = \limsup_{n \to \infty} \log \big( k_1 \cdots k_n \big) \big/ \log \big( k_n d_{n} \big)$. Assuming that
\begin{align}\label{eqn:genipcondition} \liminf_{n \to \infty} \frac{\log \big( k_1 \cdots k_n \big)}{\log \big(k_{n+1}d_{n+1}\big)} = \cdim(E),\end{align}
by associating to $\ell \geq 1$ the $n \geq 1$ for which $k_n d_n \leq \ell < k_{n+1}d_{n+1}$, we have
\begin{align*}
\cdim(E) &\geq \limsup_{\ell \to \infty} \frac{\log \big| E \cap [-\ell,\ell) \big|}{\log (2\ell)} \\
&\geq \liminf_{\ell \to \infty} \frac{\log \big| E \cap [-\ell,\ell) \big|}{\log (2\ell)} \geq \liminf_{n \to \infty} \frac{\log \big( k_1 \cdots k_n \big)}{\log \big(k_{n+1}d_{n+1}\big)} = \cdim(E).
\end{align*}
Therefore, assuming (\ref{eqn:genipcondition}), it follows that
\[\lim_{\ell \to \infty} \frac{\log \big| E \cap [-\ell,\ell) \big|}{\log (2\ell)} = \cdim(E),\]
whereby $E$ is counting universal and mass universal by Lemma \ref{lem:universality}.

\item \label{item:notuniversal} The following is a concrete example of a set in $\R$ which is counting regular but not counting universal. It is similar in spirit the set constructed in Section 4.3 of \cite{lima}. Fix $0 < \alpha < 1$. For $0 < \epsilon < \alpha$ and $N = N(\alpha,\eps) \in \N$ sufficiently large, it is not hard to check that the set $E = E(\alpha,\epsilon,N)$ of $\lfloor N^\alpha \rfloor$ points equally spaced in the interval $[0,N]$ has the following properties:
\begin{enumerate}[i.]
\item For all cubes $C$ with $1 \leq \|C\| \not\in \big[N^{\frac{1-\alpha}{1-\alpha+\epsilon}},N^{\frac{\alpha}{\alpha-\epsilon}}\big]$, $| E \cap C| \leq \|C\|^{\alpha - \epsilon}$; and
\item $\sup_{\|C\| \geq 1}  |E \cap C| \big / \|C\|^\alpha \in (1-\epsilon,1+\epsilon)$.
\end{enumerate}
Let $(\epsilon_i)_i \subseteq (0,\alpha)$, $\epsilon_i \to 0$, $(N_i)_i \subseteq \N$, $N_i \to \infty$ sufficiently fast enough that the intervals $\big[N^{\frac{1-\alpha}{1-\alpha+\epsilon}},N^{\frac{\alpha}{\alpha-\epsilon}}\big]$ in i. above are disjoint, and $E_i = E_i(\alpha,\epsilon_i,N_i) \subseteq [0,N_i]$ be the corresponding sequence of finite sets. Let $(z_i)_i \subseteq \N$ be increasing sufficiently quickly to guarantee that if $C$ is a cube intersecting both $E_i + z_i$ and $E_j + z_j$, $i \neq j$, then $|\cup_i (E_i + z_i) \cap C| \leq \|C\|^{\alpha/2}$.

Let $A = \cup_{i \text{ even}} (E_i + z_i)$ and $B = \cup_{i \text{ odd}} (E_i + z_i)$. From property ii. and the rate of increase of the $z_i$'s, both $A$ and $B$ are $\alpha$-counting regular. But, $A$ and $B$ are not counting compatible: from property i., the dimension of $B$ along any sequence of cubes along which $A$ achieves its counting dimension will be less than $\alpha - \epsilon$. It follows that neither $A$ nor $B$ is universal. It may be shown with an argument along these lines that there exists a countably infinite family of mutually counting incompatible subsets of $\R$, each of dimension 1.
\end{enumerate}
\end{examples}

Finally, we give examples of collections which are compatible. Note that if a set is (strongly) counting or mass universal, then its inclusion or exclusion from a collection of sets does not affect the (strong) counting or mass compatibility of the collection.

\begin{examples}\label{ex:compatcollections}\leavevmode
\begin{enumerate}[i.]
\item Let $A \subseteq \R$. The collection $\{A, \ldots, A\}$ is counting and mass compatible. If $A$ is counting (resp. mass) regular, then $\{A, \ldots, A\}$ is strongly counting (resp. mass) compatible. More generally, if $\{A_1, \ldots, A_m\}$ is (strongly) compatible, then repeating any number of sets in the collection results in another collection which is (strongly) compatible.
\item Any collection comprised of the set of primes, images of $\Z$ under polynomials, integer Cantor sets, and/or generalized IP sets satisfying (\ref{eqn:genipcondition}) is counting and mass compatible. Since each of these sets are counting and mass universal, any collection comprised of these sets and one other arbitrary set is also counting and mass compatible.
\item Any collection comprised of the set of images of $\Z$ under polynomials and/or integer Cantor sets is strongly counting and mass compatible. Since each of these sets are strongly counting and mass universal, any collection comprised of these sets and one other arbitrary counting and mass regular set is also strongly counting and mass compatible.
\end{enumerate}
\end{examples}

\section{Marstrand-type theorems and applications}

In this section, we prove several Marstrand-type theorems for the counting and mass dimensions and give some applications.

Marstrand's theorem is one of the fundamental theorems in geometric measure theory; it relates the Hausdorff dimension of a set and the Hausdorff dimension of the image of that set under almost all orthogonal projections. More precisely, let $E \subseteq \R^2$ be Borel, and denote by $P_\theta$ the orthogonal projection of $\R^2$ onto the line which forms an angle of $\theta$ with the $x$-axis. Marstrand proved that for Lebesgue-almost every $\theta \in [0, \pi]$,
\[\dim_H (P_\theta E) = \min \big(1, \dim_H(E) \big),\]
where $\dim_H$ denotes the Hausdorff dimension.  Marstrand's theorem was generalized to higher dimensions by Mattila \cite{mattila}.

Lima and Moreira \cite{lima} proved a Marstrand-type theorem for the counting dimension by relating the counting dimension of $A + \lambda B$ for almost every $\lambda$ to $\cdim(A) + \cdim(B)$; see Theorem B from the introduction. The following Marstrand-type theorem for the counting dimension both generalizes and strengthens this result.

\begin{theorem}\label{thm:bettermarstrand}
Let $A \subseteq \R^d$ and $0 \leq k \leq d$. If $\alpha \neq k$ is such that $\cmeas \alpha (A) > 0$, then for almost every projection $P: \R^d \to \R^d$ with range $\R^k \times \{0\}^{d-k}$, $\cmeas {\min(k,\alpha)} (PA) > 0$. In particular, for almost every such projection $P$,
\[\cdim \big( PA \big) \geq \min \big(k, \cdim(A) \big)\]
and, if $\cdim(A) > k$, then $\cmeas k (PA) > 0$.
\end{theorem}

Projections with a fixed range are parameterized by their null spaces; it is with respect to this parameterization and the rotation-invariant probability measure on the Grassmannian $G(d,d-k)$ that the statement of this theorem and the next is made.

Complementing Theorem \ref{thm:bettermarstrand} is an example showing that an increase in counting dimension is possible under the typical projection.

There is a similar Marstrand-type result for the mass dimension.

\begin{theorem}\label{thm:bettermarstrandformass}
Let $A \subseteq \R^d$ be such that $\cdim(A) = \mdim(A)$, and let $0 \leq k \leq d$. For almost every projection $P: \R^d \to \R^d$ with range $\R^k \times \{0\}^{d-k}$,
\[\mdim \big( PA \big) = \min \big(k, \mdim(A) \big)\]
and, if $A$ is counting and mass regular and $\mdim(A) \neq k$, then $\mmeas{\min (k, \mdim(A))}(P A) > 0$.
\end{theorem}

We derive two more Marstrand-type results as corollaries to these theorems. The first shows that the theorems above hold when the subspace $\R^k \times \{0\}^{d-k}$ is replaced by any other linear subspace. The second is an orthogonal projection variant reminiscent of Marstrand's original theorem and its higher dimensional generalization.

\begin{cor}\label{cor:orthogmarstrandformass}
Let $A \subseteq \R^d$ be such that $\cdim(A) = \mdim(A)$, and let $0 \leq k \leq d$. For almost every orthogonal projection $U$ with range of dimension $k$,
\[\mdim(UA) = \min \big(k, \mdim(A) \big)\]
and, if $A$ is counting and mass regular and $\mdim(A) \neq k$, then $\mmeas{\min(k,\mdim(A))}(U A) > 0$.
\end{cor}

This section is organized as follows. First we describe the spaces of projections and relevant measures on them. Then, we prove Theorem \ref{thm:bettermarstrand} and provide an example showing that an increase in counting dimension is possible under the typical projection. After proving Theorem \ref{thm:bettermarstrandformass}, we develop the measurability and non-singularity lemmas needed to deduce further Marstrand-type results. Finally, we present some applications.

Recall the asymptotic notation from Section \ref{sec:invunderqie}: $f \ll_{a,b,\ldots} g$ or $g \gg_{a,b,\ldots} f$ if $f(x) \leq K g(x)$ for all $x$ in the common domain of $f$ and $g$ (unless another domain is specified), where $K > 0$ may depend only on the quantities $a, b, \ldots$.

\subsection{Spaces of projections}\label{sec:bettermeasuresonspaces}

Let $\linmaps(\R^d)$ denote the space of linear transformations of $\R^d$. Denote by $\matspace mn = \matspace mn (\R)$ the space of real $m \times n$ matrices. The \emph{standard basis} of $\R^d$ is $\{e_i\}_{i=1}^d$ where $e_i$ is the vector with a 1 in the $i$-th coordinate and 0's elsewhere. Given a basis $B = \{b_i\}_{i=1}^d$ of $\R^d$ and a linear transformation $T \in \linmaps(\R^d)$, denote by $\matmap TB$ the matrix of $T$ with respect to $B$. The matrix of $T$ in the standard basis will be denoted $\matmape T$.


A \emph{projection} of $\R^d$ is a map $P \in \linmaps(\R^d)$ satisfying $P^2 = P$; it is \emph{orthogonal} if its range and null space are orthogonal, and otherwise it is \emph{oblique}. For a fixed $0 \leq k \leq d$ and linear subspace $R$ of $\R^d$, set
\begin{align*}
\proj dR &= \{ P \in \linmaps(\R^d) \ | \ P \text{ is a projection with } \rangespace P = R\},\\
\oproj dk &= \{ U \in \linmaps(\R^d) \ | \ U \text{ is an orthogonal projection with } \dim (\imagespace U) = k \}.
\end{align*}
We specialize the notation to $\proj dR = \proj dk$ when $R = \R^k \times \{0\}^{d-k}$. Both $\proj dR$ and $\oproj dk$ are measurable spaces with Borel $\sigma$-algebras inherited as subspaces of $\linmaps(\R^d)$ equipped with the strong operator norm induced from the usual Euclidean inner product on $\R^d$.

We will parameterize projections in $\proj dR$ and $\oproj dk$ via the Grassmannian manifolds. Denote by $G(d,k)$ the set of all $k$-dimensional linear subspaces of $\R^d$. The orthogonal group $O(d)$ acts naturally on $G(d,k)$, and through this action $G(d,k)$ may be equipped with a (unique) $O(d)$-invariant probability measure $\gamma_{d,k}$ (see Mattila \cite{mattilabook}, Chapter 3).

To parameterize the set of projections with a fixed range $R \subseteq \R^d$, set $k = \dim R$ and associate to $P \in \proj dR$ its null space $\nullspace P \in G(d,d- k)$. Let
\[G_R(d,d-k) = \big\{V \in G(d,d-k) \ | \ V \cap R = \{0\}\big\}.\]
It is a fact that $G_R(d,d-k)$ is an open subset of $G(d,d-k)$ of full $\gamma_{d,d-k}$ measure (see Mattila \cite{mattilabook}, Chapter 3). The space $\proj dR$ is in 1--1 correspondence with $G_R(d,d-k)$ and acquires a Borel probability measure $\projm dR$ from $G_R(d,d-k)$ through this correspondence. It is the space $\proj dk$ which appears in the statement of Theorems \ref{thm:bettermarstrand} and \ref{thm:bettermarstrandformass}.

The space $\oproj dk$ of orthogonal projections with range of dimension $k$ is in 1--1 correspondence with $G(d,k)$ by associating $U$ with $\rangespace U$. The space $\oproj dk$ acquires a Borel probability measure $\oprojm dk$ through this correspondence. It is this space which appears in the statement of Corollary \ref{cor:orthogmarstrandformass}.\\

It will be convenient to give a flat parameterization to $\proj dR$. To do this, we will make use of the standard atlas on $G(d,d-k)$ which gives $G(d,d-k)$ the structure of a $\big(k(d-k)\big)$-dimensional smooth manifold (see Lee \cite{leebook}). The measures on $\proj dR$ resulting from this flat parameterization will be equivalent to $\projm dR$, and so it suffices for the typicality results in question to work with these measures.

Any element $W \in G_R(d,d-k)$ is naturally identified as the graph of a linear transformation $R^\perp \to R$. Choosing bases for $R^\perp$ and $R$, there is a 1--1 correspondence between $G_R(d,d-k)$ and $\matspace k{(d-k)}$; these are precisely the charts on $G(d,d-k)$ giving it the structure of a smooth manifold. The Lebesgue measure passes to $\proj dR$ through this correspondence, and it is straightforward to show that the measures on $G_R(d,d-k)$ resulting from different choices of bases for $R^\perp$ and $R$ are mutually absolutely continuous.

Moreover, these measures are equivalent to $\gamma_{d,d-k}$ on $G_R(d,d-k)$. This follows from two facts. First, the measure $\gamma_{d,d-k}$ on $G(d,d-k)$ belongs to the smooth measure class of $G(d,d-k)$; for example, it is equivalent to the Riemannian volume of any $O(d)$-invariant Riemannian metric on $G(d,d-k)$. Second, it is a general fact that on an $n$-dimensional smooth manifold, measures from the smooth measure class are equivalent to the pull-backs of the Lebesgue measure on $\R^n$ under charts on the manifold (see, for example, Folland \cite{follandbook}, Section 11.4).

The resulting flat parameterizations of $\proj dR$ may be realized easily as follows. For $1 \leq k \leq d-1$, let $B = \{b_i\}_{i=1}^d$ be a basis for $\R^d$ whose first $k$ vectors span $R$.  Matrices of projections $P \in \proj dR$ with respect to $B$ take the form
\begin{center} $\matmap PB = $
$\left( \begin{tabular}{c|c}
$\text{Id}_{k\times k}$ & $\restmatmap PB$ \\
\hline
0 & 0
\end{tabular} \right)$
\end{center}
where $\restmatmap PB$ is a real $k \times (d-k)$ matrix, and each real $k \times (d-k)$ matrix in the form above gives a different projection in $\proj dR$; in other words, the map $\purerestmatmap B: \proj dR \rightarrow \matspace k{(d-k)}$ is a bijection. We have the following associations:

\vskip .2cm
\begin{center}\begin{tabular}{ccccccc}
$\proj dR$ & $\longleftrightarrow$ & $G_R(d,d-k)$ & $\longleftrightarrow$ & $\mathcal{L}(R^\perp,R)$ & $\longleftrightarrow$ & $\matspace k{(d-k)}$\\
$P$ & $\longleftrightarrow$ & $\nullspace P$ & $\longleftrightarrow$ & $\restr{P}{R^\perp}$ & $\longleftrightarrow$ & $\restmatmap PB$
\end{tabular}\end{center}
\vskip .05cm

It follows from the remarks above that the pull-backs of the Lebesgue measure through $\purerestmatmap B$ for varying bases $B$ and the measure $\projm dR$ on $\proj dR$ are all equivalent. We will use the most convenient of these measures when showing that maps between the spaces $\proj dR$ are non-singular in Section \ref{sec:furtherresults}.

\subsection{Transversality of oblique projections}\label{sec:keygeometriclemma}

Here we prove a key lemma on the geometry of oblique projections. Essentially, the further two points $z,z' \in \Z^d$ are apart, the smaller the measure of the set of projections in $\proj dk$ which map the points close together (see Lemma 3.11 in Mattila \cite{mattilabook}, Chapter 3, for the result for orthogonal projections).  This so-called property of \emph{transversality} is fundamental to Marstrand-type results (see \cite{peres}). 

Fix $d \geq 2$ and $1 \leq k \leq d-1$. For $M \in \matspace k{(d-k)}$, let $\pmat \in \proj dk$ be the projection which satisfies $\restmatmape {\pmat} = M$ (using notation from Section \ref{sec:bettermeasuresonspaces}). Write $\R^d = \R^k \oplus \R^{d-k}$, and let $\pone$ and $\ptwo$ be the orthogonal projections onto the first and second components, respectively; note that
\begin{align}\label{eqn:explicitformulaforproj}\pone \circ \pmat = \pone + M \circ \ptwo.\end{align}

Let $\compact \subseteq \matspace k{(d-k)}$ be compact, and for $z, z' \in \Z^d$, let
\[\compact_{z,z'} = \big\{ M \in \compact \ \big| \ \lfloor \pmat z \rfloor = \lfloor \pmat z' \rfloor \big\}\]
be the set of those $M \in \compact$ for which the projection $\pmat$ maps the points $z$ and $z'$ to the same $\Z^d$-based unit cube.

\begin{lemma}\label{lem:bestgeometry} 
For all $z, z' \in \Z^d$ with $z \neq z'$,
\[m \big(\compact_{z,z'} \big) \ll_{k,d,\compact} |z - z'|^{-k}.\]
\end{lemma}

\begin{proof}
Since the left hand side is bounded from above by $m(\compact)$ and the right hand side is positive and monotonic, it suffices to prove the lemma for $z,z' \in \Z^d$ with $|z-z'| \geq \sqrt{k}$ and $\compact_{z,z'} \neq \emptyset$.

It suffices to bound from above the measure of the set
\begin{align}\label{eqn:targetbound}S = \left\{ M \in \compact \ \big| \ \big| \pmat(z - z') \big| < \sqrt{k} \right\} \supseteq \compact_{z,z'}.\end{align}
Using (\ref{eqn:explicitformulaforproj}),
\[S = \big\{ M \in \compact \ \big| \ M \ptwo (z-z') \in B \big\},\]
where $B \subseteq \R^k$ is the open ball centered at $\pone (z'-z)$ with radius $\sqrt{k}$. Since $\compact_{z,z'}$ is non-empty, the set $S$ is non-empty. If it were that $\ptwo (z-z') = 0$, then $B$ would be an open ball containing 0 with radius $\sqrt{k}$ and center a distance $|\pone (z' -z) | = |z' - z| \geq \sqrt{k}$ from the origin, a contradiction; therefore, $\ptwo (z-z') \neq 0$.  By rotating and scaling $\ptwo (z-z')$ to $e_1 \in \R^{d-k}$, it is straightforward to show that
\begin{align}\label{eqn:easyexercise}m \big( S \big) \leq \diam \left( \compact \right)^{k(d-k-1)} \vol_k \left( \frac{B}{| \ptwo (z-z') |} \right) \ll_{k,d,\compact} \big| \ptwo (z-z') \big|^{-k},\end{align}
where $\frac{B}{| \ptwo (z-z') |}$ is the ball centered at $\frac{\pone (z'-z)}{| \ptwo (z-z') |}$ with radius $\frac{\sqrt{k}}{| \ptwo (z-z') |}$, $\diam(\compact)$ is the diameter of $\compact$ in the Euclidean metric, and $\vol_k$ is the $k$-dimensional Lebesgue measure.

The final step is to show that
\begin{align}\label{eqn:controlonmag}|z-z' | \ll_{k,d,\compact} \big| \ptwo (z-z') \big|.\end{align}
Let $M \in \compact_{z,z'}$. Since $\big| \pmat(z - z') \big| < \sqrt{k}$, we have
\[ \Big| \big| \pone (z - z') \big| - \big| M \ptwo (z-z') \big| \Big | \leq \big| \pmat(z - z') \big| < \sqrt{k}.\]
Because $\ptwo (z - z') \neq 0$ and $z,z' \in \Z^d$, we have $\big|\ptwo (z-z')\big| \geq 1$. It follows that
\[\big| \pone (z-z') \big| < \sqrt{k} + \big\| M \big\| \big|\ptwo (z - z') \big| \leq \big|\ptwo (z - z') \big| \left ( \sqrt{k} + \sup_{M \in \compact} \| M \| \right),\]
where $\|M\|$ denotes the operator norm of $M$. This shows that $\big|\pone (z-z') \big| \ll_{k,\compact} \big| \ptwo (z-z') \big|$, and (\ref{eqn:controlonmag}) follows from the Pythagorean theorem. The proof of the lemma is completed by combining (\ref{eqn:targetbound}), (\ref{eqn:easyexercise}), and (\ref{eqn:controlonmag}).
\end{proof}

\subsection{A Marstrand-type theorem for the counting dimension}\label{sec:betterspecialcase}

Here we prove the main Marstrand-type theorem for the counting dimension, Theorem C from the introduction. The proof follows the main ideas of Lima and Moreira in \cite{lima}, which constitutes the special case of Theorem \ref{thm:bettermarstrand} when $d=2$ and $k=1$ (see Theorem B from the introduction).

\begin{reptheorem}{thm:bettermarstrand}
Let $A \subseteq \R^d$ and $0 \leq k \leq d$. If $\alpha \neq k$ is such that $\cmeas \alpha (A) > 0$, then for almost every $P \in \proj dk$, $\cmeas {\min(k,\alpha)} (PA) > 0$. In particular, for almost every $P \in \proj dk$,
\[\cdim \big( PA \big) \geq \min \big(k, \cdim(A) \big)\]
and, if $\cdim(A) > k$, then $\cmeas k (PA) > 0$.
\end{reptheorem}

Fix $d \geq 2$. The conclusions of the theorem are immediate when $k = 0$ or $d$, so fix $1 \leq k \leq d-1$. It suffices to prove the first assertion in the statement of Theorem \ref{thm:bettermarstrand} since the second follows immediately from the first. Thus, we will assume in this section that $\alpha \neq k$.

By the discussion in Section \ref{sec:bettermeasuresonspaces}, it suffices to prove the statement in Theorem \ref{thm:bettermarstrand} with $P = \pmat$ (in the notation from Section \ref{sec:keygeometriclemma}) for $m$-almost every $M \in \matspace k{(d-k)}$, where $m$ is the Lebesgue measure on $\matspace k{(d-k)}$. By countable additivity, it suffices to prove the theorem for $m$-a.e. matrix in a fixed compact subset of $\matspace k{(d-k)}$. Moreover, it suffices by Corollary \ref{cor:qiecorone} and the remarks following it to prove the theorem for subsets of the integer lattice $\Z^d$. Thus, fix $A \subseteq \Z^d$ and $\compact \subseteq \matspace k{(d-k)}$ compact. In what follows, the dependence on $k$, $d$, and $\compact$ in the asymptotic notation will be suppressed.

The idea for the proof is as follows.  By Lemma \ref{lem:bestgeometry}, the further two points in $\Z^d$ are apart, the fewer number of projections there are that map the points close together. If a finite set is not too concentrated or is ``thin,'' the number of its elements should be roughly preserved (depending on how thin it is) by most projections. A subset $A \subseteq \Z^d$ of dimension $\alpha$ contains finite sets which are uniformly thin to a degree depending on $\alpha$. Most projections of $A$ will preserve the size of most of these finite subsets, and so the typical projection of $A$ will preserve the dimension of $A$.

To quantify thin-ness, we make the following definition; there are clear analogues in the continuous theory in terms of bounding the growth of measures on balls by powers of the balls' radii.

\begin{definition}\label{def:countingthin}
Let $\alpha \geq 0$ and $0 < K < \infty$.  A (possibly finite) set $E \subseteq \Z^d$ is $(K,\alpha)$\emph{-counting bounded} if for all cubes $C \subseteq \R^d$ with $\|C\| \geq 1$, $|E \cap C| \leq K \|C\|^\alpha$.
\end{definition}

\begin{remark}
Whenever a set is $(K,\alpha)$-counting bounded, it will be understood that $\alpha \geq 0$ and $0 < K < \infty$. Note that $\cmeas\alpha(A) < \infty$ if and only if $A$ is $(K,\alpha)$-counting bounded for some $K$.
\end{remark}

The goal is to show that many projections $\pmat$ roughly preserve the number of elements of a counting bounded finite set $E \subseteq \Z^d$; that is, we want to bound $\big| \lfloor \pmat E \rfloor \big|$ from below for most $M \in \compact$ (depending on the degree to which $E$ is counting bounded). In order to do this, set
\begin{align*}
\compact_{z,z'} &= \big\{ M \in \compact \ \big| \ \lfloor \pmat z \rfloor = \lfloor \pmat z' \rfloor \big\},\\
\repnum E{M}y &= \left| \big\{ z \in E \ \big| \ \lfloor \pmat z \rfloor = y \big\} \right|,\\
\repnumsum EM &= \left| \big\{ (z,z') \in E^2 \ \big| \ \lfloor \pmat z \rfloor = \lfloor \pmat z' \rfloor \big\} \right|\\
&= \sum_{y \in \Z^{d}} \repnum E{M} y^2 = \sum_{z,z' \in E} \chi_{\compact_{z,z'}}(M),\\
\Delta(E) & = \int_\compact \repnumsum EM \ dm(M) = \sum_{z,z' \in E} m \left(\compact_{z,z'} \right),
\end{align*}
where $\chi_{\compact_{z,z'}}$ is the indicator function of the set $\compact_{z,z'}$. Bounding $\Delta(E)$ from above will allow us to bound $\repnumsum EM$ from above for most $M \in \compact$, which in turn by the Cauchy-Schwarz inequality will yield a lower bound on $\big| \lfloor \pmat E \rfloor \big|$. (If $E  = E_1 \times E_2 \subseteq \Z^2$, then set $S_E(M)$ is the so-called additive energy between $E_1$ and $M E_2$. Bounding the additive energy from above in order to obtain a lower bound on the size of a sumset via Cauchy-Schwarz is a well known technique; see, for example, Tao-Vu \cite{taovu}, Corollary 2.10.)

\begin{definition}
For a (metrically) bounded subset $B \subseteq \R^d$, denote by $\|B\|$ the \emph{cubic diameter} of $B$, that is, the infimum of $\|C\|$ over all cubes $C$ containing $B$.
\end{definition}

The following lemma gives the desired upper bound on $\Delta(E)$ from the transversality property in Lemma \ref{lem:bestgeometry}.

\begin{lemma}\label{lem:bettercontrolonpairs}
For all $\alpha \neq k$ and all finite $(K,\alpha)$-counting bounded sets $E \subseteq \Z^d$,
\[\Delta(E) \ll_{K,\alpha} |E|\|E\|^{\max(0,\alpha-k)}.\]
\end{lemma}

\begin{proof}
Let $E \subseteq \Z^d$ be a finite, $(K,\alpha)$-counting bounded set and $L = \lfloor \log \diam(E) \rfloor$. Write
\[\Delta(E) = \sum_{z,z' \in E} m \left(\compact_{z,z'} \right) = |E|m \left( \compact \right) + \sum_{z \in E} \sum_{\ell = 0}^{L} \sum_{\substack{z' \in E \\ e^\ell \leq |z-z'| < e^{\ell + 1}}} m \left(\compact_{z,z'} \right),\]
where the sum is split into the diagonal and off-diagonal terms.  Since $|E|m \left( \compact \right) \ll |E|$, it suffices to show the desired bound on the off-diagonal term.  Fix $z \in E$ and $\ell \geq 0$. Since $E$ is $(K,\alpha)$-counting bounded,
\[\big| \left\{ z' \in E \ \middle| \ e^\ell \leq |z-z'| < e^{\ell + 1} \right\} \big| \ll_{K,\alpha} \left( e^\ell \right)^\alpha,\]
and for each such $z'$, since $|z-z'| \geq e^\ell$, Lemma \ref{lem:bestgeometry} gives that $m(\compact_{z,z'}) \ll \left(e^\ell \right)^{-k}$. Therefore, summing on $\ell$ and using the fact that $\alpha \neq k$,
\[\sum_{\ell = 0}^{L} \sum_{\substack{z' \in E \\ e^\ell \leq |z-z'| < e^{\ell + 1}}} m \left(\compact_{z,z'} \right) \ll_{K,\alpha} \sum_{\ell = 0}^{L} \left( e^{\alpha - k} \right)^\ell \ll_\alpha \|E\|^{\max(0,\alpha-k)}.\]
Since this bound is independent of $z$, the result follows by summing over $z \in E$.
\end{proof}

The following proposition quantifies the claim that \emph{most projections preserve the cardinality of a counting bounded finite set} by establishing the equivalent statement \emph{few projections do not preserve the cardinality of a counting bounded finite set}.

In interpreting the following result, consider $\|E\|^\alpha$ as being approximately $|E|$. If $\alpha < k$, then the measure of the set of projections which reduce $E$ to $\delta |E|$ elements (after rounding) is proportional to $\delta$. If $\alpha > k$, then the (rounded) image of $E$ must contain fewer elements since it lies in $\Z^{k}$; indeed, the measure of the set of projections which reduce $E$ to $\delta |E|^{\frac{k}{\alpha}}$ elements is proportional to $\delta$.

\begin{prop}\label{prop:bettersmallprojs}
For all $\alpha \neq k$, all finite $(K,\alpha)$-counting bounded sets $E \subseteq \Z^d$, and all $\delta >0$,
\[m \left( \left\{ M \in \compact \ \middle| \ \big| \lfloor \pmat E \rfloor \big | < \delta \frac{|E|}{\|E\|^{\max(0,\alpha-k)}} \right\} \right) \ll_{K,\alpha} \delta.\]
\end{prop}

\begin{remark}\label{rem:betterjustdependsonthin}
Note that the upper bound established here depends only the degree to which $E$ is counting bounded; in particular, it is independent of $E$ itself.
\end{remark}

\begin{proof}
Let $E \subseteq \Z^d$ be finite and $(K,\alpha)$-counting bounded. Since
\[\int_\compact S_E(M) \, dm(M) = \Delta(E),\]
for all $\eps > 0$,
\[m\left( \left\{ M \in \compact \ \middle | \ S_E(M) \leq \frac{\Delta(E)}{\eps} \right\} \right) \geq m(\compact) - \eps.\]
If $M \in \compact$ is such that $S_E(M) \leq \eps^{-1} \Delta(E)$, then by Cauchy-Schwarz,
\[\big| \lfloor \pmat E \rfloor \big| \geq \frac{ \left( \sum_{y \in \Z^d} R_{E, M} (y) \right)^2} {\sum_{y \in \Z^{d}} R_{E, M} (y)^2} = \frac{|E|^2}{S_E(M)} \geq \frac{|E|^2}{\eps^{-1}\Delta(E)}.\]
Therefore,
\[m\left( \left\{ M \in \compact \ \middle| \ \big| \lfloor \pmat E \rfloor \big| < \frac{|E|^2}{\eps^{-1}\Delta(E)} \right\} \right) < \eps.\]
Lemma \ref{lem:bettercontrolonpairs} gives that $\Delta(E) \ll_{K,\alpha} |E|\|E\|^{\max(0,\alpha-k)}$, so the result follows by setting $\eps$ equal to the product of $\delta$ and the asymptotic constant.
\end{proof}

We may now prove Theorem \ref{thm:bettermarstrand}.

\begin{proof}
By the monotonicity of the counting measures, it suffices to show that $\cmeas {\min(k,\alpha)} (\pmat A') > 0$ for some $A' \subseteq A$, so, by passing to a subset of $A$ via Proposition \ref{prop:regularsubset}, we may assume that $A$ is $\alpha$-counting regular. Let $(C_n)_n \subseteq \R^d$, $\|C_n\| \rightarrow \infty$, be a sequence of cubes for which
\[0 < \cmeas\alpha(A) = \lim_{n \rightarrow \infty} \frac{|A \cap C_n|}{\|C_n\|^\alpha} < \infty.\]
Let $A_n = A \cap C_n$. Since $\cmeas\alpha(A) < \infty$, $A$ is $(K, \alpha)$-counting bounded for some $0 < K < \infty$, and each $A_n$ is finite and $(K, \alpha)$-counting bounded.

For $\delta > 0$ and $n \in \N$, let
\begin{align}
\label{eqn:betterbigproj} G_\delta^n &= \left\{ M \in \compact \ \middle| \ \big| \lfloor \pmat A_n \rfloor \big| \geq \delta \frac{|A_n|}{\|A_n\|^{\max(0,\alpha - k)}} \right\},\\
\notag G_\delta &= \bigcap_{m =1}^\infty \bigcup_{n = m}^\infty G_\delta^n.
\end{align}
The set $G_\delta$ consists of those $M \in \compact$ for which the inequality in (\ref{eqn:betterbigproj}) holds for infinitely many $n \in \N$. By Proposition \ref{prop:bettersmallprojs}, $m(\compact \setminus G_\delta^n) \ll_{K,\alpha} \delta$, where the bound here is independent of $n$ (see Remark \ref{rem:betterjustdependsonthin}). Since
\[\compact \setminus \bigcap_{m =1}^\infty \bigcup_{n = m}^\infty G_\delta^n = \bigcup_{m =1}^\infty \bigcap_{n = m}^\infty \left( \compact \setminus G_\delta^n \right),\]
we have that $m(\compact \setminus G_\delta) \ll_{K,\alpha} \delta$.

Let $M \in G_\delta$. It follows from (\ref{eqn:explicitformulaforproj}) and the fact that $\compact$ is bounded that $\|\pmat C_n \| \ll \|C_n\|$. Therefore, for each $n \in \N$, there exists a cube $C_n'$ in $\R^d$ containing $\lfloor \pmat C_n \rfloor$ with $\|C_n'\| \ll \|C_n\|$. It follows that
\begin{align}
\notag \cmeas{\min(k,\alpha)}(\pmat A) &\geq \limsup_{n \rightarrow \infty} \frac{\big| \lfloor \pmat A \rfloor \cap C_n' \big|}{\|C_n'\|^{\min(k,\alpha)}} \\
\label{eqn:bettersecond}&\gg \limsup_{n \rightarrow \infty} \frac{\big| \lfloor \pmat A_n \rfloor \big|}{\|C_n\|^{\min(k,\alpha)}}\\
\label{eqn:betterthird}& \geq \limsup_{n \rightarrow \infty} \frac{\delta |A_n| }{\|C_n\|^{\min(k,\alpha)}\|A_n\|^{\max(0,\alpha - k)}}\\
\label{eqn:betterfourth}& \geq \delta \limsup_{n \rightarrow \infty} \frac{|A \cap C_n| }{\|C_n\|^{\alpha}} = \delta \cmeas\alpha(A) > 0,
\end{align}
where (\ref{eqn:bettersecond}) follows since $C_n'$ contains $\lfloor \pmat C_n \rfloor$ and $\|C_n'\| \ll \|C_n\|$, line (\ref{eqn:betterthird}) follows since $M \in G_\delta$, and line (\ref{eqn:betterfourth}) follows since $\|A_n\| \leq \|C_n\|$.

Let $G = \bigcup_{m=1}^\infty G_{m^{-1}}$; by the work above, $m(G) = m(\compact)$, and for all $M \in G$, $\cmeas{\min(k,\alpha)}(\pmat A) > 0$. In other words, for almost every $M \in \compact$, $\cmeas{\min(k,\alpha)}(\pmat A) > 0$.
\end{proof}

\subsection{Increase in counting dimension}\label{sec:dimincrease}

Unlike in the case of Marstrand's original theorem, there are sets which exhibit an increase in counting dimension under the typical projection. In this section, we give an example of such a set.

For $\lambda \in \R$, let $P_\lambda: \R^2 \to \R^2$ be the projection with range $\R \times \{0\}$ and null space $\big\langle (\lambda, -1) \big\rangle$. By associating the range with $\R$ and by a slight abuse of notation, we will write simply that $P_\lambda (a,b) = a+\lambda b$.

\begin{prop}
There exist sets $A \subseteq \R$ and $B \subseteq \Z$ with $\cdim(A) = \cdim(B) = 0$ such that for all $\lambda \neq 0$, the set $\big\lfloor P_\lambda (A \times B) \big\rfloor$ is thick in $\Z$, that is, contains arbitrarily long intervals.
\end{prop}

\begin{proof}
For $n \in \N$, let $b_n = n^n$, and set $B = (b_n)_n$. Let $v: \N \to \N$ be a sequence with the property that for all $n \in \N$, the set $v^{-1}(n)$ is infinite (for example, $v(n) = \nu_2(n)+1$, where $\nu_2$ is the 2-adic valuation). Let $\eps_n = {(2n)}^{-2n}$, and let $(\lambda_i)_i \subseteq (-\infty,0)$ be such that for all $n \in \N$,
\begin{align}\label{eqn:lambdaprop}\big\{ \lambda_i \ \big| \ v(i) = n \big\} \text{ is contained in and is } \eps_n\text{-dense in } \left(-\infty,-n^{-1} \right).\end{align}
We now construct finite sets $A_1, A_2, \ldots \subseteq \R$ inductively so that $A' = \cup_i A_i$ has counting dimension 0 and such that for all $\lambda < 0$, the set $\big\lfloor P_\lambda(A' \times B) \big\rfloor$ is thick. The sets $A = A' \cup (-A')$ and $B$ will then satisfy the conclusions of the proposition.

Choose $1 < a_1^1 < \cdots < a_{v(1)}^1 \in \R$ so that for some $n_1 \in \Z$, for all $1 \leq i \leq v(1)$,
\[P_{\lambda_1} \big( (a_i^1,b_{v(1)+i-1}) \big) = n_1 + i - \frac{1}{2}.\]
(This can be accomplished by intersecting the horizontal lines passing through points of $\{0\} \times B$ with lines of slope $-\lambda_1^{-1}$ passing through half-integer points on $\R \times \{0\}$.) We claim that the set $A_1 = \{a^1_i\}_{i = 1}^{v(1)}$ has the following properties:
\begin{enumerate}[i.]
\item $\sup_{\|C\| \geq 1} \frac{\log |A_1 \cap C|}{\log \|C\|} < v(1)^{-1}$, and
\item for all $\lambda \in \left( \lambda_1 - \eps_{v(1)}, \lambda_1 + \eps_{v(1)} \right)$,
\[\big\lfloor P_{\lambda}(A_1) \big\rfloor = \big\lfloor P_{\lambda_1}(A_1) \big\rfloor = \big\{n_1, \ldots, n_1 + v(1) - 1\big\}.\]
\end{enumerate}

To see i., note that for $1 \leq i \leq v(1)$,
\[n_1 + i - \frac{1}{2} = P_{\lambda_1}\big( ( a_i^1, b_{v(1)+i-1}) \big) = a_i^1 + \lambda_1 b_{v(1)+i-1}.\]
It follows from this, the definition of the $b_n$'s, and (\ref{eqn:lambdaprop}) that
\[a_{i+1}^1 - a_i^1 = 1- \lambda_1 (b_{v(1)+i} - b_{v(1)+i-1}) > \frac{b_{v(1)+1} - b_{v(1)}}{v(1)} > b_{v(1)}.\]
Therefore, if $C \subseteq \R$ is a cube with $\|C\| \geq 1$ and $m = |A_1 \cap C | \geq 2$, then $\|C\| \geq (m-1) b_{v(1)}$ and
\[\frac{\log |A_1 \cap C|}{\log \|C\|} \leq \frac{\log m}{\log (m-1) + v(1) \log v(1)} \leq \frac{1}{v(1)}.\]

To see ii., note that for any $\lambda \in \R$,
\begin{align*}
P_{\lambda}\big( ( a_i^1, b_{v(1)+i-1}) \big) &= P_{\lambda_1}\big( ( a_i^1, b_{v(1)+i-1}) \big) + (\lambda - \lambda_1) b_{v(1)+i-1}\\
&= n_1 + i - \frac{1}{2} + (\lambda - \lambda_1) b_{v(1)+i-1}.
\end{align*}
Therefore, if for all $1 \leq i \leq v(1)$,
\begin{align}\label{eqn:lambdarange}\lambda - \lambda_1 < \frac{1}{2 b_{v(1)+i-1}},\end{align}
then $\lfloor P_\lambda A_1 \rfloor = \lfloor P_{\lambda_1} A_1 \rfloor$. Since $\eps_{v(1)} = {(2v(1))}^{-2v(1)} < (2 b_{2v(1)-1})^{-1}$, equation (\ref{eqn:lambdarange}) holds for all $\lambda \in \left( \lambda_1 - \eps_{v(1)}, \lambda_1 + \eps_{v(1)} \right)$.

Suppose now that $A_i \subseteq \R$ and $n_i \in \R$ have been defined for all $1 \leq i \leq m$ in such a way that properties i. and ii. hold (with $A_i, \lambda_i, v(i)$, and $n_i$ replacing $A_1, \lambda_1, v(1)$, and $n_1$), and so that
\begin{enumerate}[i.]
\setcounter{enumi}{2}
\item $\displaystyle \sup_{\|C\| \geq 1} \frac{\log \big|\cup_{i=1}^m A_i \cap C \big|}{\log \|C\|} \leq \max_{i=1,\ldots,m} \sup_{\|C\| \geq 1} \frac{\log | A_i \cap C|}{\log \|C\|}$.
\end{enumerate}
(Property iii. is achievable by choosing at each stage $a_1^i$ much larger than $a_{v(i-1)}^{i-1}$ so as to separate sufficiently $A_{i}$ from $A_{i-1}$.)

Choose $a_1^{m+1} \in \R$ sufficiently large, and choose $a_1^{m+1} < a_2^{m+1} < \cdots < a_{v(m+1)}^{m+1} \in \R$ so that for some $n_{m+1} \in \Z$, for all $1 \leq i \leq v(m+1)$,
\[P_{\lambda_{m+1}} \big( (a_i^{m+1},b_{v(m+1)+i-1}) \big) = n_{m+1} + i - \frac{1}{2}.\]
Properties i. and ii. hold for $A_{m+1} = \{a^{m+1}_i\}_{i = 1}^{v(m+1)}$ by the same reasoning as above, and property iii. holds if $a_1^{m+1}$ was chosen large enough.

Set $A' = \cup_i A_i$. It follows from properties i. and iii. of the sets $A_1, A_2, \ldots$ above that $\cdim(A') = 0$, and it follows from property ii. and (\ref{eqn:lambdaprop}) that for all $\lambda < 0$, the set $\lfloor P_\lambda (A' \times B) \rfloor$ is thick. By symmetry, the sets $A = A' \cup (-A')$ and $B$ satisfy the conclusions of the proposition.
\end{proof}

Using Corollary \ref{cor:qiecorone} and the fact that $\big\lfloor P_\lambda (A \times B) \big\rfloor = \lfloor A+\lambda B \rfloor$ is thick for all $\lambda \neq 0$, we may conclude that $E = \lfloor A \rfloor \cup B$ is a set of integers with $\cdim(E) = 0$ for which there exists an $\eps > 0$ such that for all $\lambda \neq 0$,
\[\cmeas1 \big( E + \lambda E \big) = d^* \big( \lfloor E + \lambda E \rfloor \big) \geq \eps.\]
In other words, even for sets of integers, the generic sumset $\lambda_1 A + \lambda_2 B$ may have dimension strictly greater than $\cdim(A) + \cdim(B)$.

It is worth mentioning here an example from Section 4.3 of \cite{lima} addressing the other extreme.  There exist sets $A, B \subseteq \Z$ with $\cdim(A) + \cdim(B) > 1$ for which $d^*(A + \lambda B) = 0$ for all $\lambda \in \R$. 

\subsection{A Marstrand-type theorem for the mass dimension}

In this section, we prove the main Marstrand-type result for the mass dimension. In contrast to the situation for the counting dimension, a decrease, but not an increase, in mass dimension is possible under the typical projection of a set $A \subseteq \R^d$. A Marstrand-type result is recovered under the additional assumption that $\mdim(A) = \cdim(A)$.

\begin{reptheorem}{thm:bettermarstrandformass}
Let $A \subseteq \R^d$ be such that $\cdim(A) = \mdim(A)$, and let $0 \leq k \leq d$. For almost every $P \in \proj dk$,
\[\mdim \big( PA \big) = \min \big(k, \mdim(A) \big)\]
and, if $A$ is counting and mass regular and $\mdim(A) \neq k$, then $\mmeas{\min (k, \mdim(A))}(P A) > 0$.
\end{reptheorem}

The proof is separated into two bounds for $\mdim(PA)$ for the typical projection $P \in \proj dk$: a lower bound (Proposition \ref{prop:relatemassandcounting}) similar to that in Theorem \ref{thm:bettermarstrand} and an upper bound (Proposition \ref{prop:upperboundondim}) showing that an increase in mass dimension does not occur. Theorem \ref{thm:bettermarstrandformass} follows immediately upon combining the two results.

\begin{prop}\label{prop:relatemassandcounting}
Let $A \subseteq \R^d$ and $0 \leq k \leq d$. For almost every $P \in \proj dk$,
\[\mdim(P A) \geq \min \big( \mdim(A), \mdim(A) - \cdim(A) + k \big).\]
If $A$ is counting and mass regular and $\cdim(A) = \mdim(A) \neq k$, then for almost every $P \in \proj dk$, $\mmeas{\min (k, \mdim(A))}(P A) > 0$.
\end{prop}

As in the proof of Theorem \ref{thm:bettermarstrand}, it suffices to fix $d \geq 2$, $1 \leq k \leq d-1$, and, by Corollary \ref{cor:qiecorone}, $A \subseteq \Z^d$ and prove the statement with $P = P_M$ for $m$-almost every $M$ in a fixed compact subset $\compact$ of $\matspace k{(d-k)}$. Again, dependence on $k$, $d$, and $\compact$ in the asymptotic notation will be suppressed.

The analogous notion of thin-ness is useful here; recall Definition \ref{def:countingthin}.

\begin{definition}\label{def:massthin}
Let $\alpha \geq 0$ and $0 < K < \infty$.  A (possibly finite) set $E \subseteq \Z^d$ is $(K,\alpha)$-\emph{mass bounded} if for all centered cubes $C \subseteq \R^d$ with $\|C\| \geq 1$, $|E \cap C| \leq K \|C\|^\alpha$.
\end{definition}

\begin{remark}
Whenever a set is $(K,\alpha)$-mass bounded, it will be understood that $\alpha \geq 0$ and $0 < K < \infty$. Note that $\mmeas\alpha(A) < \infty$ if and only if $A$ is $(K,\alpha)$-mass bounded for some $K$.
\end{remark}

Now we may prove Proposition \ref{prop:relatemassandcounting}; the proof uses tools from Section \ref{sec:betterspecialcase} and is similar to the proof of Theorem \ref{thm:bettermarstrand}.

\begin{proof}
Let $\alpha > \cdim(A)$, $\alpha \neq k$. Let $(C_n)_n \subseteq \R^d$ be a sequence of centered cubes along which $A$ achieves its mass dimension, and let $A_n = A \cap C_n$. Since $\cmeas\alpha(A) = 0$, the set $A$ is $(K, \alpha)$-counting bounded for some $0 < K < \infty$, and each $A_n$ is finite and $(K, \alpha)$-counting bounded.

For $\delta > 0$ and $n \in \N$, let $G_\delta^n$ and $G_\delta$ be defined as in (\ref{eqn:betterbigproj}). It follows just as before that $m(\compact \setminus G_\delta) \ll_{K,\alpha} \delta$.

Let $M \in G_\delta$. It follows from (\ref{eqn:explicitformulaforproj}) and the fact that $\compact$ is bounded that $\|\pmat C_n \| \ll \|C_n\|$. Since $\pmat$ is linear, for each $n \in \N$, there exists a centered cube $C_n'$ such that $\lfloor \pmat C_n \rfloor \subseteq C_n'$ and $\|C_n'\| \ll \|C_n\|$. Then,
\begin{align}
\notag \mdim(\pmat A) &\geq \limsup_{n \rightarrow \infty} \frac{\log \big| \lfloor \pmat A \rfloor \cap C_n' \big|}{\log \|C_n'\|}\\
\label{eqn:bettersecondtwo} &\geq \limsup_{n \rightarrow \infty} \frac{\log \big| \lfloor \pmat A_n \rfloor \big|}{\log \|C_n\|}\\
\label{eqn:betterthirdtwo} & \geq \limsup_{n \rightarrow \infty} \frac{\log \big( \delta |A_n| \|A_n\|^{-\max (0,\alpha-k )} \big) }{\log \|C_n\|}\\
\notag &= \limsup_{n \rightarrow \infty} \left( \frac{\log \delta}{\log \|C_n\|} + \frac{\log |A_n|}{\log \|C_n\|} - \max\big(0,\alpha-k \big) \frac{\log \|A_n\|}{\log \|C_n\|} \right)\\
\label{eqn:betterfourthtwo} & \begin{aligned} \, =& \lim_{n \rightarrow \infty} \left( \frac{\log \delta}{\log \|C_n\|} \right) + \lim_{n \rightarrow \infty} \left( \frac{\log |A \cap C_n|}{\log \|C_n\|} \right) \\ &\qquad \qquad \qquad \qquad + \min\big(0,-\alpha+k \big) \liminf_{n \rightarrow \infty} \left( \frac{\log \|A_n\|}{\log \|C_n\|} \right)\end{aligned}\\
\label{eqn:betterfifthtwo} & \geq \mdim(A) + \min\big(0,-\alpha+k \big),
\end{align}
where (\ref{eqn:bettersecondtwo}) follows since $C_n'$ contains $\lfloor \pmat C_n \rfloor$ and $\|C_n'\| \ll \|C_n\|$; line (\ref{eqn:betterthirdtwo}) follows since $M \in G_\delta$; line (\ref{eqn:betterfourthtwo}) follows since the second limit in exists since $A$ was assumed to achieve its mass dimension along $(C_n)_n$, and the $\limsup$ becomes $\liminf$ by the sign of its coefficient; and line (\ref{eqn:betterfifthtwo}) follows since $\|A_n\| \leq \|C_n\|$.

If $G = \bigcup_{m=1}^\infty G_{m^{-1}}$, then $m(G) = m(\compact)$, and so for almost every $M \in \compact$,
\[\mdim(\pmat A) \geq \min\big(\mdim(A),\mdim(A)-\alpha+k \big).\]
Since $\alpha > \cdim(A)$, $\alpha \neq k$, was arbitrary, for almost every $M \in \compact$,
\[\mdim(\pmat A) \geq \min\big(\mdim(A),\mdim(A)-\cdim(A)+k \big).\]

If $A$ is both counting and mass regular and $\cdim(A) = \mdim(A) \neq k$, set $\alpha = \mdim(A)$ and choose a sequence of centered cubes $(C_n)_n$ so that $A$ achieves its mass measure along it. Since $\cmeas\alpha(A) < \infty$, there is a $0 < K < \infty$ so that each $A_n = A \cap C_n$ is finite and $(K, \alpha)$-counting bounded. Proceeding now as in the proof of Theorem \ref{thm:bettermarstrand}, we may conclude that for almost every $M \in \compact$, $\mmeas{\min (k, \mdim(A))}(\pmat A) > 0$.
\end{proof}

The assumption of counting and mass regularity is necessary in the proof above due to the fact that sets in general do not admit subsets which are both counting and mass regular (see Example \ref{ex:nosimultsubset}).

We now turn to bounding the mass dimension of the typical projection from above. This reduces easily in the continuous setting to the fact that projection maps are Lipschitz, but there is no such analogue here. In general, there may be many directions in which the mass dimension of a set increases under projection. Indeed, it is an exercise to construct a set $A \subseteq \Z^2$ with $\mdim(A) = 0$ (actually, $\cdim(A)=0$ just as easily) for which $\lfloor P_\lambda A \rfloor = \Z$ for all $\lambda \in \mathbb{Q}$ (in the notation of Section \ref{sec:dimincrease}). However, it is still true that the mass dimension does not increase under the typical projection.

\begin{prop}\label{prop:upperboundondim}
Let $A \subseteq \R^d$ and $0 \leq k \leq d$. For almost every $P \in \proj dk$,
\[\mdim \big( P A \big) \leq \min \big( k, \mdim(A) \big).\]
\end{prop}

Again, fix $d \geq 2$, $1 \leq k \leq d-1$, $A \subseteq \Z^d$, and $\compact \subseteq \matspace k{(d-k)}$ compact. In what follows, let $I_n^d$ be the centered cube $\left[-\frac n2, \frac n2 \right)^d$. Dependence on $k$, $d$, and $\compact$ in the asymptotic notation will be suppressed.

We need two lemmas, the first of which is an easy estimation argument whose proof is left to the reader.

\begin{lemma}\label{lem:dimalongexpcubes}
For all $A \subseteq \R^d$,
\[\mdim(A) = \limsup_{j \to \infty} \frac{\log \big| \lfloor A \rfloor \cap I_{2^j}^d \big|}{\log 2^j}.\]
\end{lemma}

\begin{lemma}\label{lem:lemmaforupperbound}
Let $E \subseteq \Z^d$ be $(K,\alpha)$-mass bounded and $0 \leq k \leq d$. If $\alpha < k$, then for all $n \in \N$ and for all $\eps > 0$,
\[m \left(\left\{ M \in \compact \ \middle| \ \frac{\log \big|\lfloor \pmat E \rfloor \cap I_n^d \big|}{\log n} > \alpha + \eps \right\}\right) \ll_{K,\alpha} n^{-\eps}.\]
\end{lemma}

\begin{proof}
Let $\Is_n^k = I_n^k \times \{0\}^{d-k}$. The first step is to prove that for all $n \in \N$,
\begin{align}\label{eqn:keyfirststep}\sum_{z \in \Is_n^k} \sum_{z' \in E} m(\compact_{z,z'}) \ll_{K,\alpha} n^{\alpha}.\end{align}
Since the double sum is increasing in $n$, it suffices to prove the statement for all $n \geq 2 \sqrt{k}$. For $i \in \N$, set $I_i' = I_{2^i n}^d$ and $S_i' = I_i' \setminus I_{i-1}'$; note that $S_1' = I_1'$. Since $E$ is $(K,\alpha)$-mass bounded,
\begin{align}\label{eqn:eqn1}|E \cap S_i'| \leq |E \cap I_i'| \leq K \|I_i'\|^\alpha \ll_{K} \left(2^{ i } n \right)^\alpha.\end{align}
By interchanging sums and splitting the sum on $E$ over the shells, we may write
\[\sum_{z \in \Is_n^k} \sum_{z' \in E} m(\compact_{z,z'}) = \sum_{i = 1}^\infty \sum_{z' \in E \cap S_i'} \sum_{z \in \Is_n^k} m(\compact_{z,z'}).\]
When $i = 1$, by the definition of $\compact_{z,z'}$,
\[\sum_{z' \in E \cap S_1'} \sum_{z \in \Is_n^k} m(\compact_{z,z'}) \leq \sum_{z' \in E \cap S_1'} m(\compact) \leq m(\compact) \big|E \cap S_1' \big| \ll_{K,\alpha} n^\alpha.\]
Let $i \geq 2$. Note that each coordinate of a point $z' \in E \cap S_i'$ is at least $2^{i-2}n$ in magnitude. It follows that for all $z \in \Is^k_n$,
\[|z-z'| \geq 2^{i-3}n,\]
and so, by Lemma \ref{lem:bestgeometry},
\[m( \compact_{z,z'} ) \ll \left(2^i n \right)^{-k}.\]
Since $\alpha < k$,
\begin{align*}\
\sum_{i = 2}^\infty \sum_{z' \in E \cap S_i'} \sum_{z \in \Is_n^k} m(\compact_{z,z'}) &\ll \sum_{i = 2}^\infty \big|E \cap S_i' \big| \left|\Is_n^k \right| \left(2^i n \right)^{-k}\\
&\ll_K \sum_{i = 2}^\infty \left(2^{ i } n \right)^\alpha n^{k} \left(2^i n \right)^{-k}\\
&= n^\alpha \sum_{i = 2}^\infty \left(2^{\alpha-k} \right)^i \ll_{\alpha} n^\alpha.
\end{align*}
Along with the $i=1$ case, this establishes (\ref{eqn:keyfirststep}).

Now, define $f: \compact \to \R$ by
\[f(M) = \left| \lfloor \pmat E \rfloor \cap I_n^d \right| = \left| \lfloor \pmat E \rfloor \cap \Is_n^k \right| = \sum_{z \in \Is_n^k} \big[ z \in \lfloor \pmat E \rfloor \big],\]
where the Iverson brackets $[\textsc{expression}]$ return 1 if \textsc{expression} is true and 0 otherwise; this map is measurable by Lemma \ref{lem:measurablemap}. By the definition of $f$ and (\ref{eqn:keyfirststep}),
\[\|f\|_{L^1(\compact,m)} = \sum_{z \in \Is_n^k} m\left( \bigcup_{z' \in E} \compact_{z,z'} \right) \leq \sum_{z \in \Is_n^k} \sum_{z' \in E} m(\compact_{z,z'}) \ll_{K,\alpha} n^{\alpha}.\]
By Chebyshev's inequality,
\[m \big( \left \{ M \in \compact \ \middle| \ f(M) > n^{\alpha + \eps} \right\} \big) \leq n^{-(\alpha + \eps)} \|f\|_1 \ll_{K,\alpha} n^{-\eps}.\]
Finally, the left hand side of the previous expression is exactly the expression in the conclusion by taking logarithms.
\end{proof}

We now combine the lemmas above to prove Proposition \ref{prop:upperboundondim}.

\begin{proof}
Since $\pmat A \subseteq \R^k \times \{0\}^{d-k}$, we have $\mdim(\pmat A) \leq k$, so there is nothing more to show if $\mdim(A) \geq k$. Suppose $\mdim(A) < k$, and let $\mdim(A) < \alpha < k$. It suffices by Lemma \ref{lem:dimalongexpcubes} and the countable additivity of $m$ to show that for all $\eps > 0$, the set
\[L_{\eps} = \left\{ M \in \compact \ \middle| \ \limsup_{j \to \infty} \frac{\log \big| \lfloor \pmat A \rfloor \cap I_{2^j}^{d} \big|}{\log 2^j} > \alpha + \eps \right\}\]
has zero measure. By the definition of the limit supremum, for all $J \in \N$,
\[L_\eps(J) = \bigcup_{j \geq J} \left\{ M \in \compact \ \middle| \ \frac{\log \big| \lfloor \pmat A \rfloor \cap I_{2^j}^{d} \big|}{\log 2^j} > \alpha + \eps \right\} \supseteq L_{\eps},\]
so it suffices to show that $m\big(L_\eps (J)\big) \to 0$ as $J \to \infty$. Since $\alpha > \mdim(A)$, the set $A$ is $(K,\alpha)$-mass bounded for some $0 < K < \infty$. Since $\alpha < k$, Lemma \ref{lem:lemmaforupperbound} applies to give that
\[m \big( L_\eps (J) \big) \ll_{K,\alpha} \sum_{j \geq J} \left(2^j \right)^{-\eps}.\]
The right hand side is the tail of a convergent series, so it tends to $0$ as $J \to \infty$.
\end{proof}

The following example complements Theorem \ref{thm:bettermarstrandformass} by showing that the assumption $\cdim(A) = \mdim(A)$ is necessary to prevent a decrease in mass dimension under the typical projection. As before, for $\lambda \in \R$, let $P_\lambda: \R^2 \to \R^2$ be the projection with range $\R \times \{0\}$ and null space $\big\langle (\lambda, -1) \big\rangle$.

\begin{example}
Let $0 < \alpha < 1$. There exists a set $A \subseteq \R^2$ with $\mdim(A) = 2\alpha$ with the property that for all $\lambda \in \R$, $\mdim(P_\lambda A) = \alpha$. To construct such a set, for each $\ell \in \N$, consider the cube
\[A_\ell = \big[\ell-\ell^\alpha,\ell \big) \times \big[0,\ell^\alpha \big).\]
Consider a sequence $(\ell_n)_n$ of positive real numbers which increases very rapidly, and let $A = \cup_n A_{\ell_n}$. The set $A \subseteq \R^2$ achieves a mass dimension of at least $2 \alpha$ along the sequence of centered cubes with side lengths $\ell_n$, and if $(\ell_n)_n$ is increasing rapidly enough, $A$ will have mass dimension $2 \alpha$. Moreover, for each fixed $\lambda \in \R$, the set $P_\lambda A$ is a union of intervals, one of length approximately $(\ell_n)^\alpha$ at a distance of approximately $\ell_n$ away from the origin for each $n \in \N$. If the sequence $(\ell_n)_n$ is increasing rapidly enough, the set $P_\lambda A$ achieves a mass dimension of $\alpha$.
\end{example}

\subsection{Further Marstrand-type results}\label{sec:furtherresults}

In this section, we derive the analogous Marstrand-type results for projections with a fixed range $R \subseteq \R^d$ and orthogonal projections with range of dimension $k$. Each of these results requires that certain maps be measurable and non-singular with respect to the measures described in Section \ref{sec:bettermeasuresonspaces}.

\begin{lemma}\label{lem:measurablemap}
Let $A \subseteq \R^d$. The maps $\linmaps (\R^d)$ to $\R$ defined by $T \mapsto \cmeas\alpha(TA)$ and $T \mapsto \cdim(TA)$ are measurable.

The same conclusions hold for the maps involving the mass measures and dimension.
\end{lemma}

\begin{proof}
Recall that
\[\cmeas\alpha(TA) = \lim_{\ell \rightarrow \infty} \sup_{\|C\| \geq \ell} \frac{\big| \lfloor TA \rfloor \cap C \big|}{\|C\|^\alpha}.\]
By taking the limit $\ell \rightarrow \infty$ along $\ell \in \N$ and considering only those cubes with integral side length based at points in $\Z^d$, it suffices to show that for a fixed such cube $C$, the map
\[T \mapsto \big| \lfloor TA \rfloor \cap C \big|\]
is measurable. If $C = \cup_i C_i$ is a partition of $C$ into unit cubes, then $\big| \lfloor TA \rfloor \cap C \big| = \sum_i \big| \lfloor TA \rfloor \cap C_i \big|$. Therefore, it suffices to assume that $C$ is a unit cube and prove that $T \mapsto \big| \lfloor TA \rfloor \cap C \big| = \big[ TA \cap C \neq \emptyset \big]$ is measurable (here, $[\, \cdot \, ]$ denotes the Iverson brackets).

If $E \subseteq \R^d$ is open (resp. closed), then the preimage of 1 (resp. 0) under $T \mapsto \big[ TA \cap E \neq \emptyset \big]$ is open, hence measurable. It follows that $T \mapsto \big[ TA \cap E \neq \emptyset \big]$ is measurable. Since the cube $C$ is a union of disjoint open and closed sets, the map $T \mapsto \big[ TA \cap C \neq \emptyset \big]$ is measurable.

That $T \mapsto \cdim(TA)$ is measurable follows from the same work using the explicit formula (\ref{eqn:explicitformula}) for $\cdim(TA)$ from Lemma \ref{lem:basiccountingprops}.
\end{proof}

\begin{definition}
A measurable map $\varphi: (X,\mathcal{B},\mu) \longrightarrow (Y,\mathcal{C},\nu)$ is \emph{non-singular} if for all $C \in \mathcal{C}$, $\nu(C) = 0$ if and only if $\mu(\varphi^{-1} C) = 0$. Equivalently, $\varphi$ is non-singular if $\varphi_\# \mu$ is equivalent to $\nu$, where of $\varphi_\# \mu$ is the push-forward of $\mu$ to $Y$ through $\varphi$.
\end{definition}

In order to deduce analogues of the main Marstrand-type theorems for the space of projections $\proj dR$, we will rotate the subspace $R$ to the subspace $\R^k \times \{0\}^{d-k}$. The following lemma gives that the identification of good projections before and after this rotation is non-singular.

\begin{lemma}\label{lem:measmapslemma1}
Let $R \subseteq \R^d$ be a linear subspace of $\R^d$ of dimension $k$, $1 \leq k \leq d-1$. Let $B = \{b_i\}_{i=1}^d$ be a basis of $\R^d$ whose first $k$ vectors span $R$, and let $T$ be the linear map defined by $Te_i = b_i$, $1 \leq i \leq d$. The map
\begin{gather*}
\mynextmap dR: \proj dR \longrightarrow \proj dk \\
P \mapsto T^{-1} \circ P \circ T
\end{gather*}
is measurable and non-singular.
\end{lemma}

\begin{proof}
The map $\mynextmap dR$ is continuous, hence it is measurable. Let $\myrestnextmap dR: \matspace k{(d-k)} \rightarrow \matspace k{(d-k)}$ be the map making the following diagram commutative:

\[\begin{CD}
\proj dR @>\mynextmap dR>> \proj dk\\
@VV \purerestmatmap B V @VV \purerestmatmape V\\
\matspace k{(d-k)} @>\myrestnextmap dR>>\matspace k{(d-k)}
\end{CD}\]
\vskip .2cm

It suffices to prove that $\myrestnextmap dR$ is non-singular. It follows from the definition of $\mynextmap dR$ that $\myrestnextmap dR$ is defined for $M \in \matspace k{(d-k)}$ by the following equation (all matrices are $d \times d$):
\begin{center}
$\matmape {T^{-1}} B \left( \begin{tabular}{c|c}
$\text{Id}_{k\times k}$ & $M$ \\
\hline
0 & 0
\end{tabular} \right) B^{-1} \matmape T = \left( \begin{tabular}{c|c}
$\text{Id}_{k\times k}$ & $\myrestnextmap dR (M)$ \\
\hline
0 & 0
\end{tabular} \right)$,
\end{center}
where, by a slight abuse of notation, $B$ is the matrix with columns $b_1, \ldots, b_d$. By the definition of $T$, $[T] = B$, and so $\myrestnextmap dR (M) = M$, which is clearly non-singular.
\end{proof}

\begin{cor}\label{cor:obliquemarstrand}
Let $A \subseteq \R^d$ and $R \subseteq \R^d$ be a linear subspace of dimension $k$. If $\cmeas \alpha (A) > 0$ for some $\alpha \neq k$, then for almost every $P \in \proj dR$, $\cmeas {\min(k,\alpha)} (PA) > 0$. In particular, for almost every $P \in \proj dR$,
\[\cdim \big( PA \big) \geq \min \big(k, \cdim(A) \big)\]
and, if $\cdim(A) > k$, then $\cmeas {k} (PA) > 0$.
\end{cor}

\begin{proof}
If $k = 0$ or $d$, the conclusion is immediate, so suppose $1 \leq k \leq d-1$. Let $\{b_1,\ldots,b_d\}$ be a basis for $\R^d$ whose first $k$ vectors span $R$, and let $T$ be the linear map defined by $Te_i = b_i$, $1 \leq i \leq d$.

Suppose $\cmeas \alpha (A) > 0$ for some $\alpha \neq k$. Since $T$ is an invertible linear transformation, by Corollary \ref{cor:qiecortwo}, we have $\cmeas \alpha (T^{-1}A) > 0$. Set
\begin{align*}
\goo_R &= \big\{ P \in \proj dR \ \big| \ \cmeas {\min(k,\alpha)} (PA) > 0 \big\},\\
\goo &= \big\{ P \in \proj dk \ \big| \ \cmeas {\min(k,\alpha)} \big(PT^{-1}A \big) > 0 \big\}.
\end{align*}
The sets $\goo_R$ and $\goo$ are measurable by Lemma \ref{lem:measurablemap}, and $\goo$ is of full measure in $\proj dk$ by Theorem \ref{thm:bettermarstrand}.

Let $\mynextmap dR: \proj dR \rightarrow \proj dk$ be as in Lemma \ref{lem:measmapslemma1}. Since $\mynextmap dR$ is non-singular and $\goo$ is of full measure, to show that $\goo_R$ is of full measure, it suffices to prove that $\mynextmap dR ^{-1} \goo \subseteq \goo_R$.

If $P \in \mynextmap dR ^{-1} \goo$, then $T^{-1} \circ P \circ T \in \goo$, meaning
\[\cmeas {\min(k,\alpha)} \big(T^{-1}P A \big) = \cmeas {\min(k,\alpha)} \big(T^{-1}P T T^{-1}A \big) > 0.\]
By Corollary \ref{cor:qiecortwo}, this implies that $\cmeas {\min(k,\alpha)} \big(P A \big) > 0$, meaning $P \in \goo_R$.
\end{proof}

The proof of the analogous corollary for the mass dimension follows in exactly the same way from Theorem \ref{thm:bettermarstrandformass} and Corollary \ref{cor:qiecortwo}.

\begin{cor}\label{cor:obliquemarstrandformass}
Let $A \subseteq \R^d$ be such that $\cdim(A) = \mdim(A)$, and let $R \subseteq \R^d$ be a linear subspace of dimension $k$. For almost every $P \in \proj dR$,
\[\mdim\big( PA \big) = \min \big(k, \mdim(A) \big)\]
and, if $A$ is counting and mass regular and $\mdim(A) \neq k$, then $\mmeas{\min(k,\mdim(A))}(P A) > 0$.
\end{cor}

In order to deduce analogues of the main Marstrand-type theorems for orthogonal projections, we will associate an oblique projection $P$ with range $\R^k \times \{0\}^{d-k}$ to the orthogonal projection with range $\nullspace P ^\perp$. The images of a set under $P$ and the associated orthogonal projection differ by an invertible linear transformation; Corollary \ref{cor:qiecortwo} applies to give that the dimensions of these images are equal.

\begin{lemma}\label{lem:equivparam}
Let $0 \leq k \leq d$. The map $\obtoorth: \proj dk \longrightarrow \oproj dk$ sending $P$ to the orthogonal projection with range $\nullspace P ^\perp$ is a measurable, non-singular bijection (modulo null sets).
\end{lemma}

\begin{proof}
This is immediate from the discussion in Section \ref{sec:bettermeasuresonspaces} and the fact that $V \mapsto V^\perp$ is a measure preserving (hence, non-singular) bijection between $G(d,k)$ and $G(d,d-k)$.
\end{proof}

\begin{cor}\label{cor:orthogmarstrand}
Let $A \subseteq \R^d$ and $0 \leq k \leq d$. If $\cmeas \alpha (A) > 0$ for some $\alpha \neq k$, then for almost every $U \in \oproj dk$, $\cmeas {\min(k,\alpha)} (UA) > 0$. In particular, for almost every $U \in \oproj dk$,
\[\cdim \big( UA \big) \geq \min \big(k, \cdim(A) \big)\]
and, if $\cdim(A) > k$, then $\cmeas k (UA) > 0$.
\end{cor}

\begin{proof}
Let $\obtoorth: \proj dk \longrightarrow \oproj dk$ be as in Lemma \ref{lem:equivparam}. Let $P \in \proj dk$, and set $U = \obtoorth (P)$, $R = \nullspace P ^\perp = \imagespace U$. Since $\restr{P}{R}: R \to \R^k \times \{0\}^{d-k}$ is an invertible linear map, Corollary \ref{cor:qiecortwo} gives that for all $B \subseteq R$,
\[\cmeas \alpha (B) \asymp_{\alpha,d,P} \cmeas \alpha (PB).\]
Since $\imagespace U = R$ and $P\circ U = P$, for all $B \subseteq \R^d$,
\begin{align}\label{eqn:equivoforthandoblique} \cmeas \alpha (UB) \asymp_{\alpha,d,P} \cmeas \alpha (PUB) = \cmeas \alpha (PB).\end{align}

Suppose $\cmeas \alpha (A) > 0$ for some $\alpha \neq k$. Set
\begin{align*}
\goo_{\mathcal{U}} &= \big\{ U \in \oproj dk \ \big| \ \cmeas {\min(k,\alpha)} (UA) > 0 \big\},\\
\goo &= \big\{ P \in \proj dk \ \big| \ \cmeas {\min(k,\alpha)} (PA) > 0 \big\}.
\end{align*}
The sets $\goo_{\mathcal{U}}$ and $\goo$ are measurable by Lemma \ref{lem:measurablemap}, and $\goo$ is of full measure in $\proj dk$ by Theorem \ref{thm:bettermarstrand}.

Since $\obtoorth$ is non-singular and $\goo$ is of full measure, to show that $\goo_{\mathcal{U}}$ is of full measure, it suffices to prove that $\goo \subseteq \obtoorth ^{-1} (\goo_{\mathcal{U}})$. So, let $P \in \goo$. It follows from $\cmeas {\min(k,\alpha)} (PA) > 0$ and (\ref{eqn:equivoforthandoblique}) that $\cmeas {\min(k,\alpha)} \big(\obtoorth (P)A\big) > 0$. Therefore, $\obtoorth (P) \in \goo_{\mathcal{U}}$.
\end{proof}

As before, the proof of the analogous corollary for the mass dimension follows in the same way from Theorem \ref{thm:bettermarstrandformass} and Corollary \ref{cor:qiecortwo}.

\begin{repcor}{cor:orthogmarstrandformass}
Let $A \subseteq \R^d$ be such that $\cdim(A) = \mdim(A)$, and let $0 \leq k \leq d$. For almost every $U \in \oproj dk$,
\[\mdim(UA) = \min \big(k, \mdim(A)\big)\]
and, if $A$ is counting and mass regular and $\mdim(A) \neq k$, then $\mmeas{\min(k,\mdim(A))}(U A) > 0$.
\end{repcor}

\subsection{Applications}

The main application of Theorems \ref{thm:bettermarstrand} and \ref{thm:bettermarstrandformass} is arithmetic, a result of the fact that projections of the product set $A_1 \times \cdots \times A_d \subseteq \R^d$ to $\R \times \{0\}^{d-1}$ are (naturally identified with) sets of the form
\[A_1 + \lambda_1 A_2 + \cdots + \lambda_{d-1} A_d = \big\{ a_1 + \lambda_1 a_2 + \cdots + \lambda_{d-1} a_d \ \big| \ a_i \in A_i \big\}\]
for $\lambda = (\lambda_1, \ldots, \lambda_{d-1}) \in \R^{d-1}$. Understanding the dimension and regularity of $A_1 \times \cdots \times A_d$ gives us information on the dimension and regularity of such sumsets for Lebesgue almost every $\lambda$.

Corollaries \ref{cor:mainapplication} and \ref{cor:mainapplicationformass} below follow immediately from the definitions in Section \ref{sec:products} and from Theorems \ref{thm:bettermarstrand} and \ref{thm:bettermarstrandformass}, respectively. By dilating the set $A_1 + \lambda_1 A_2 + \cdots + \lambda_{d-1} A_d$ by a non-zero constant, the corollaries below may be stated so that the sumset in question is in the homogeneous form $\lambda_1 A_1 + \lambda_2 A_2 + \cdots + \lambda_d A_d$.

The following Corollary is similar to, but more general than, Theorem 1.3 in \cite{lima}.

\begin{cor}\label{cor:mainapplication}
Let $A_i \subseteq \R$, $1 \leq i \leq d$. If $\{A_1, \ldots, A_d\}$ is counting compatible, then for Lebesgue-a.e. $\lambda \in \R^d$,
\[\cdim(\lambda_1 A_1 + \cdots + \lambda_d A_d) \geq \min \big(1, \cdim(A_1) + \cdots + \cdim(A_d) \big)\]
and, if $\sum_i \cdim(A_i) > 1$, $\cmeas 1 (\lambda_1 A_1 + \cdots + \lambda_d A_d) > 0$. If $\{A_1, \ldots, A_d\}$ is strongly counting compatible and $\sum_i \cdim(A_i) \neq 1$, then for Lebesgue-a.e. $\lambda \in \R^d$,
\[\cmeas{\min(1, \sum_i \cdim(A_i))} (\lambda_1 A_1 + \cdots + \lambda_d A_d) > 0.\]
\end{cor}

\begin{cor}\label{cor:mainapplicationformass}
For each $1 \leq i \leq d$, let $A_i \subseteq \R$ be such that $\cdim(A_i) = \mdim(A_i)$. If $\{A_1, \ldots, A_d\}$ is counting and mass compatible, then for Lebesgue-a.e. $\lambda \in \R^d$,
\[\mdim(\lambda_1 A_1 + \cdots + \lambda_d A_d) = \min \big(1, \mdim(A_1) + \cdots + \mdim(A_d) \big).\]
If $\{A_1, \ldots, A_d\}$ is strongly counting and strongly mass compatible and $\sum_i \mdim(A_i) \neq 1$, then for Lebesgue-a.e. $\lambda \in \R^d$,
\[\mmeas{\min(1, \sum_i \mdim(A_i))} (\lambda_1 A_1 + \cdots + \lambda_d A_d) > 0.\]
\end{cor}

The following more concrete examples follow immediately from the previous two corollaries using Examples \ref{ex:compatcollections} from Section \ref{sec:products}.

\begin{examples}
The statements below hold for Lebesgue-a.e. $\lambda \in \R^d$, where the set of exceptional $\lambda$'s depends on the specific sets in each example.
\begin{enumerate}[i.]
\item For $P$ the set of prime numbers and $f$ a real, non-constant polynomial,
\[d^* \big(\lfloor P + \lambda f(\Z) \rfloor \big) > 0.\]

\item For $A \subseteq \R$,
\[\cdim(\lambda_1 A + \cdots + \lambda_d A) \geq \min \big(1, d \cdim(A) \big).\]
Moreover,
\begin{enumerate}[a.]
\item if $d \cdim(A) > 1$, then $\cmeas 1 (\lambda_1 A + \cdots + \lambda_d A) > 0$;
\item if $A$ is counting regular and $dD(A) \neq 1$, then $\cmeas{\min \left(1, d \cdim(A) \right)}(\lambda_1 A + \cdots + \lambda_d A) > 0$;
\item if $\cdim(A) = \mdim(A)$, then $\mdim(\lambda_1 A + \cdots + \lambda_d A) = \min \big(1, d \mdim(A) \big)$;
\item if $\cdim(A) = \mdim(A)$, $A$ is counting and mass regular, and $d\mdim(A) \neq 1$, then $\cmeas{\min \left(1, d \mdim(A) \right)}(\lambda_1 A + \cdots + \lambda_d A) > 0$.
\end{enumerate}

\item For non-constant polynomials $f_1, \ldots, f_d \in \R[x]$,
\[\mdim \big( \lambda_1 f_1(\Z) + \cdots + \lambda_d f_d(\Z) \big) = \min \left(1, \frac{1}{\deg f_1} + \cdots + \frac{1}{\deg f_d} \right).\]
Moreover, if $\sum_i (\deg f_i)^{-1} \neq 1$, then
\[\mmeas{\min \left(1,\sum_i (\deg f_i)^{-1} \right)} \big(\lambda_1 f_1(\Z) + \cdots + \lambda_d f_d(\Z) \big) > 0.\]
This example in the special case that $f_i \in \Z[x]$ is essentially Theorem 1.1 in \cite{lima}.

\item This final example is meant to demonstrate the generality of the results in this paper using examples from Section \ref{sec:products}. Let $A \subseteq \R$. The integer Cantor set $C$ consisting of non-negative integers that may be written in base 7 using only the digits 0 and 6 has counting and mass dimension $\log 2 \big / \log 7$ and is counting and mass universal. The generalized IP set
\[E = \left\{ \sum_{i=1}^n x_i 2^{i^2} \ \middle| \ n \in \N, \ 0 \leq x_i < 2^i \right\}\]
has counting and mass dimension $1/2$ and is counting and mass universal since it satisfies (\ref{eqn:genipcondition}). It follows that
\[\cdim(\lambda_1 A + \lambda_2 C + \lambda_3 E) \geq \min \left(1, \cdim(A) + \frac{\log 2}{\log 7} + \frac{1}{2} \right).\]
Moreover, if $\cdim(A) = \mdim(A)$, then
\[\mdim(\lambda_1 A + \lambda_2 C + \lambda_3 E) = \min \left(1, \mdim(A) + \frac{\log 2}{\log 7} + \frac{1}{2} \right).\]
\end{enumerate}
\end{examples}

\bibliographystyle{plain}
\bibliography{marstrand-type_theorems}

\end{document}